\documentclass[microtype]{gtpart}     % Basic GT/GTM/AGT style
\usepackage{graphicx}  %%% the recommended graphics package
\usepackage[all]{xy}

\usepackage{mathtools,dsfont,nicefrac} %packages for additional symbols and math typesetting
\usepackage[lite]{amsrefs}%for references
\usepackage{enumerate}%more options for the enumerate environment
\usepackage[super]{nth}%superscripts for ordinal numbers

\usepackage{verbatim} %for source code and multiline comments

\title{$\textit{L}^{\text{2}}$-invariants of nonuniform lattices\\ in semisimple Lie groups}

\author{Holger Kammeyer}
\givenname{Holger}
\surname{Kammeyer}
\address{Mathematisches Institut\\ Universit\"at Bonn\\ Endenicher Allee 60\\ 53115 Bonn\\ Germany}
\email{kammeyer@math.uni-bonn.de}
\urladdr{http://www.math.uni-bonn.de/people/kammeyer/}


\arxivreference{TODO}
\arxivpassword{TODO}

\volumenumber{}
\issuenumber{}
\publicationyear{}
\papernumber{}
\startpage{}
\endpage{}
\doi{}
\Zbl{}
\received{}
\revised{}
\accepted{}
\published{}
\publishedonline{}
\proposed{}
\seconded{}
\corresponding{}
\editor{}
\version{}

\newtheorem{theorem}{Theorem}[section]
\newtheorem{corollary}{Corollary}[section]
\newtheorem{lemma}{Lemma}[section]
\newtheorem{proposition}{Proposition}[section]
\newtheorem{conjecture}{Conjecture}[section]
\newtheorem{question}{Question}[section]
\theoremstyle{remark}
\newtheorem{example}{Example}[section]
\theoremstyle{definition}
\newtheorem{definition}{Definition}[section]

\makeautorefname{subsection}{Section}

\makeatletter
   \let\c@corollary=\c@theorem
   \let\c@lemma=\c@theorem
   \let\c@proposition=\c@theorem
   \let\c@condition=\c@theorem
   \let\c@conjecture=\c@theorem
   \let\c@question=\c@theorem
   \let\c@problem=\c@theorem
   \let\c@example=\c@theorem
   \let\c@remark=\c@theorem
   \let\c@definition=\c@theorem
\makeatother

\newcommand*{\MRref}[2]{ \href{http://www.ams.org/mathscinet-getitem?mr=#1}{MR \textbf{#1}}}
\newcommand*{\arXivref}[1]{ \href{http://www.arxiv.org/abs/#1}{arXiv:\textbf{#1}}}

\newcommand*{\fr}[1]{\mathfrak{#1}}%Lie algebra Fraktur
\newcommand*{\ad}{\textup{ad}}%Lie algebra adjoint action
\newcommand*{\rank}{\textup{rank}}%rank
\newcommand*{\Gg}{\mathbf G}%algebraic group
\newcommand*{\Hh}{\mathbf H}%another algebraic group
\newcommand*{\Pp}{\mathbf P}%parabolic subgroup
\newcommand*{\Qq}{\mathbf Q}%another parabolic subgroup
\newcommand*{\Ll}{\mathbf L}%Levi subgroup
\newcommand*{\Ss}{\mathbf S}%Q-split torus
\newcommand*{\Nn}{\mathbf N}%unipotent radical
\newcommand*{\Mm}{\mathbf M}%complement of Q-split center

\makeatletter
\def\moverlay{\mathpalette\mov@rlay}
\def\mov@rlay#1#2{\leavevmode\vtop{%
   \baselineskip\z@skip \lineskiplimit-\maxdimen
   \ialign{\hfil$\m@th#1##$\hfil\cr#2\crcr}}}
\newcommand{\charfusion}[3][\mathord]{
    #1{\ifx#1\mathop\vphantom{#2}\fi
        \mathpalette\mov@rlay{#2\cr#3}
      }
    \ifx#1\mathop\expandafter\displaylimits\fi}
\makeatother

\newcommand{\bigcupdot}{\charfusion[\mathop]{\bigcup}{\cdot}}
\numberwithin{equation}{section}
\numberwithin{figure}{section}

\begin{document}

\begin{abstract}    % type your abstract below
We compute \(L^2\)-invariants of certain nonuniform lattices in semisimple Lie groups by means of the Borel-Serre compactification of arithmetically defined locally symmetric spaces. The main results give new estimates for Novikov--Shubin numbers and vanishing \(L^2\)-torsion for lattices in groups with even deficiency. We discuss applications to Gromov's Zero-in-the-Spectrum Conjecture as well as to a proportionality conjecture for the \(L^2\)-torsion of measure equivalent groups.
\end{abstract}

\maketitle

%%%%%%%%%%%%%%%%%%%%   Start of main body of article

\section{Introduction}

Let \(\Gamma\) be a discrete countable group and consider a finite free  \(\Gamma\)-CW complex \(X\) with cellular chain complex \(C_p(X)\).  The group \(\Gamma\) acts isometrically on the \(L^2\)-completion \(C^{(2)}_p(X) = \ell^2 (\Gamma) \otimes_{\Z \Gamma} C_p(X)\) and the differentials of \(C_p(X)\) induce the \(L^2\)-Laplacian \(\Delta_p = d_p^* d_p + d_{p+1} d_{p+1}^*\) acting on \(C^{(2)}_p(X)\).  \(L^2\)-invariants of \(X\) capture spectral properties of the bounded \(\Gamma\)-equivariant operators \(\Delta_p\).

The \emph{\(L^2\)-Betti numbers} \(b^{(2)}_p(X) = \dim_{\mathcal{N}(\Gamma)} \ker \Delta_p\) for \(p \geq 0\) form the simplest example.  Their definition involves the real valued \emph{von Neumann dimension} induced by the trace of the group von Neumann algebra \(\mathcal{N}(\Gamma)\).  It turns out that \(L^2\)-Betti numbers provide powerful invariants.  As an example, a positive \(L^2\)-Betti number obstructs nontrivial circle actions and nontrivial self-coverings of \(\Gamma \backslash X\).  We will be concerned with two more sophisticated types of \(L^2\)-invariants.  The \(p\)-th \emph{Novikov--Shubin invariant} \(\widetilde{\alpha}_p(X) \in [0, \infty] \cup \infty^+\) measures by von Neumann dimension how slowly aggregated eigenspaces of \(\Delta_p\) grow for small eigenvalues.  The \emph{\(L^2\)-torsion} \(\rho^{(2)}(X) \in \R\) is the \(L^2\)-counterpart of classical Reidemeister torsion.  It is only defined if \(X\) is \emph{\(\det\)-\(L^2\)-acyclic} which essentially means that \(b^{(2)}_p(X) = 0\) for \(p \geq 0\).

\(L^2\)-invariants are homotopy invariants so that we immediately obtain invariants of groups with finite \(E\Gamma\).  An important class of those groups is given by torsion-free lattices \(\Gamma \subset G\) in semisimple Lie groups.  If such a lattice is uniform (has compact quotient), a finite \(E\Gamma\) is given by the \emph{symmetric space} \(X = G/K\) where \(K \subset G\) is a maximal compact subgroup.  The \emph{locally symmetric space} \(\Gamma \backslash X\) is then a closed manifold model for \(B\Gamma\).  This opens up the analytic approach to \(L^2\)-invariants where the cellular \(L^2\)-Laplacian is replaced by the de Rham Laplacian acting on differential \(p\)-forms of X.  With this method the \(L^2\)-invariants of all uniform lattices have been computed by M.\,Olbrich \cite{Olbrich:L2-InvariantsLocSym} building on previous work by E.\,Hess--T.\,Schick, J.\,Lott and A.\,Borel, see \fullref{thm:l2symmetricspaces}.

It is however fairly restrictive to require that lattices be uniform as this already rules out the most natural example, \(\textup{SL}(n,\Z)\), which is central to number theory and geometry.  Therefore the purpose of this paper is to calculate \(L^2\)-invariants of nonuniform lattices by using a compactification of the locally symmetric space \(\Gamma \backslash X\).  Of course the compactification has to be homotopy equivalent to the original \(\Gamma \backslash X\) to make sure it is a \(B\Gamma\).  A construction due to A.\,Borel and J.-P.\,Serre suggests to add boundary components at infinity so that \(\Gamma \backslash X\) forms the interior of a compact manifold with \emph{corners}.  To expand on this, let us first suppose that \(\Gamma\) is \emph{irreducible} and that \(G\) is connected linear with \(\rank_\R G > 1\).  Then G.\,Margulis' celebrated arithmeticity theorem says we may assume there exists a semisimple linear algebraic \(\Q\)-group \(\Gg\) such that \(G = \Gg^0(\R)\) and such that \(\Gamma\) is \emph{commensurable} with \(\Gg(\Z)\).  We assemble certain nilmanifolds \(N_P\) and so-called \emph{boundary symmetric spaces} \(X_\Pp = M_\Pp / K_P\) to \emph{boundary components} \(e(\Pp) = N_P \times X_\Pp\) associated with the rational parabolic subgroups \(\Pp \subset \Gg\).  We define a topology on the \emph{bordification} \(\overline{X} = \bigcup_\Pp e(\Pp)\) such that \(e(\Qq)\) is contained in the closure of \(e(\Pp)\) if and only if \(\Qq \subset \Pp\).  The \(\Gamma\)-action on \(X = e(\Gg)\) extends freely to \(\overline{X}\).  The bordification \(\overline{X}\) is still contractible but now has a compact quotient \(\Gamma \backslash \overline{X}\) called the \emph{Borel--Serre compactification} of the locally symmetric space \(\Gamma \backslash X\).

We note that \(L^2\)-Betti numbers and Novikov-Shubin invariants have been defined for groups with not necessarily finite \(E\Gamma\) in \citelist{\cite{Cheeger-Gromov:L2-cohomology} \cite{Lueck:ArbitraryModules} \cite{Lueck-Reich-Schick:ArbitraryNovikovShubin}}.  In the case of \(L^2\)-Betti numbers it follows already from the work of J.\,Cheeger--M.\,Gromov \cite{Cheeger-Gromov:BoundsVonNeumann} that for a lattice \(\Gamma \subset G\), uniform or not, \(b^{(2)}_p(\Gamma) \neq 0\) if and only if \(\dim X = 2p\) and \(\delta(G) = 0\) where \(\delta(G) = \rank_\C (G) - \rank_\C(K)\) is the \emph{deficiency} of \(G\).  To the author's knowledge, the only results for Novikov--Shubin invariants and \(L^2\)-torsion of nonuniform lattices have been obtained in the hyperbolic case.  An upper bound for the first Novikov-Shubin invariant of compact hyperbolic 3-manifolds was given by J.\,Lott and W.\,L\"uck  \cite{Lott-Lueck:L2ThreeManifolds}.  This can be seen as the case \(\Gg = \textup{SO}(3,1; \C)\) of our first result.

\begin{theorem} \label{thm:novikovshubinqrankone}
Let \(\Gg\) be a connected semisimple linear algebraic \(\Q\)-group.  Suppose that \(\rank_\Q(\Gg) = 1\) and \(\delta(\Gg(\R)) > 0\).  Let \(\Pp \subset \Gg\) be a proper rational parabolic subgroup.  Then for every arithmetic lattice \(\Gamma \subset \Gg(\Q)\)
\[ \widetilde{\alpha}_q(\Gamma) \leq \delta(M_\Pp) + d(N_P). \]
\end{theorem}

Here \(q\) is the \emph{middle dimension} of \(X\), so either \(\dim X = 2q\) or \(\dim X = 2q + 1\), and \(d(N_P)\) denotes the degree of polynomial growth of the unipotent radical \(N_P\) of \(P = \Pp(\R)\).  An important feature of the theorem is that no restriction is imposed on the real rank of \(\Gg\).  We will for example construct lattices \(\Gamma \subset \textup{SL}(4,\R)\) that fall under the assumptions of the theorem so that \(\widetilde{\alpha}_4(\Gamma) \leq 4\).

The \(L^2\)-torsion of a torsion-free lattice \(\Gamma \subset G\) is only defined if \(\overline{X}\) is \(\det\)-\(L^2\)-acyclic which is equivalent to \(\delta(G) > 0\).  The only such rank one Lie groups without compact factors are the groups \(G = \textup{SO}^0(2n+1,1)\).  It is a deep result of W.\,L\"uck--T.\,Schick \cite{Lueck-Schick:Hyperbolic} that the \(L^2\)-torsion of a torsion-free lattice \(\Gamma \subset \textup{SO}^0(2n+1,1)\) is proportional to the hyperbolic covolume, the first few proportionality constants being \(-\frac{1}{6\pi}\), \(\frac{31}{45\pi^2}\) and \(-\frac{221}{70\pi^3}\) for \(n = 1, 2, 3\).  One can get rid of the torsion-freeness assumption by defining the \emph{virtual \(L^2\)-torsion} \(\rho^{(2)}_{\textup{virt}}(\Gamma) = \frac{\rho^{(2)}(\Gamma')}{[\Gamma : \Gamma']}\) for a torsion free subgroup \(\Gamma' \subset \Gamma\) of finite index which always exists by Selberg's Lemma \cite{Alperin:SelbergsLemma}.  This is well-defined because \(L^2\)-torsion is multiplicative under finite coverings.  In contrast to the result of L\"uck--Schick, we prove that for higher rank Lie groups the virtual \(L^2\)-torsion vanishes in (at least) half of all cases.

\begin{theorem} \label{thm:virtuall2torsion}
Let \(G\) be a connected semisimple linear Lie group with positive, even deficiency.  Then every lattice \(\Gamma \subset G\) is virtually \(\det\)-\(L^2\)-acyclic and
\[ \rho^{(2)}_{\textup{virt}}(\Gamma) = 0. \]
\end{theorem}

Note that this is a statement about higher rank Lie groups because \(\rank_\R(G) \geq \delta(G) \geq 2\).  For example \(\rho^{(2)}_{\textup{virt}}(\textup{SL}(n;\Z)) = 0\) if \(n > 2\) and \(n = 1 \text{ or } 2 \!\mod 4\).

The computation of \(L^2\)-invariants is a worthwhile challenge in itself.  Yet we want to convince the reader that the problem is not isolated within the mathematical landscape.  The following conjecture goes back to M.\,Gromov.  We state it in a version that appears in \cite{Lueck:L2-Invariants}*{p.\,437}.

\begin{conjecture}[Zero-in-the-spectrum Conjecture] \label{conj:zerointhespectrum}
Let \(M\) be a closed aspherical Riemannian manifold.  Then there is \(p \geq 0\) such that zero is in the spectrum of the minimal closure of the Laplacian
\[ (\Delta_p)_{\textup{min}}\co \textup{dom}((\Delta_p)_{\textup{min}}) \subset L^2 \Omega^p (\widetilde{M}) \rightarrow L^2 \Omega^p(\widetilde{M}) \]
acting on \(p\)-forms of the universal covering \(\widetilde{M}\) with the induced metric.
\end{conjecture}

The conjecture has gained interest due to its relevance for seemingly unrelated questions.  For one example, the zero-in-the-spectrum conjecture for \(M\) with \(\Gamma = \pi_1(M)\) is a consequence of the \emph{strong Novikov conjecture} for \(\Gamma\) which in turn is contained in the \emph{Baum--Connes conjecture} for \(\Gamma\).  Following the survey \cite{Lueck:L2-Invariants}*{Chapter 12}, let us choose a \(\Gamma\)-triangulation \(X\) of \(\widetilde{M}\).  We define the homology \(\mathcal{N}(\Gamma)\)-module \(H^\Gamma_p(X; \mathcal{N}(\Gamma)) = H_p(\mathcal{N}(\Gamma) \otimes_{\Z \Gamma} C_*(X))\) where we view the group von Neumann algebra \(\mathcal{N}(\Gamma)\) as a discrete ring.  Then the zero-in-the-spectrum conjecture has the equivalent algebraic version that for some \(p \leq \dim M\) the homology \(H^\Gamma_p(X; \mathcal{N}(\Gamma))\) does not vanish.  \(L^2\)-invariants enter the picture in that for a general finite \(\Gamma\)-CW complex \(X\) we have \(H^\Gamma_p(X; \mathcal{N}(\Gamma)) = 0\) for \(p \geq 0\) if and only if \(b^{(2)}_p(X) = 0\) and \(\widetilde{\alpha}_p(X) = \infty^+\) for \(p \geq 0\).  Therefore Olbrich's \fullref{thm:l2symmetricspaces} implies that closed locally symmetric spaces \(\Gamma \backslash X\) coming from uniform lattices satisfy the conjecture.  The statement of the conjecture does not immediately include locally symmetric spaces \(\Gamma \backslash X\) coming from nonuniform lattices because they are not compact.  Therefore W.\,L\"uck has asked the following more general question, see \cite{Lueck:L2-Invariants}*{p.\,440}.

\begin{question} \label{quest:lueck}
If a group \(\Gamma\) has a finite CW-model for \(B\Gamma\), is there \(p \geq 0\) such that \(H^\Gamma_p(E\Gamma; \mathcal{N}(\Gamma))\) does not vanish?
\end{question}

Now this question makes sense for nonuniform lattices, and as we said, \(L^2\)-Betti numbers and Novikov--Shubin invariants provide a way to answer it.  In our case \fullref{thm:novikovshubinqrankone} combined with \fullref{thm:l2bettiofgenerallattices} by Cheeger--Gromov gives the following result.

\begin{theorem}
The answer to \textup{\fullref{quest:lueck}} is affirmative for torsion-free arithmetic subgroups of connected semisimple linear algebraic \(\Q\)-groups \(\Gg\) with \(\rank_\Q(\Gg) = 1\).
\end{theorem}

In a different direction, D.\,Gaboriau has proven in the far-reaching paper \cite{Gaboriau:Invariantsl2} that if \(\Gamma\) and \(\Lambda\) are measure equivalent groups of index \(c\) in the sense of M.\,Gromov, then \(b^{(2)}_p(\Gamma) = c \cdot b^{(2)}_p(\Lambda)\).  For obvious reasons nothing similar can be true for Novikov--Shubin invariants but for the \(L^2\)-torsion we have the following conjecture.

\begin{conjecture}[L\"uck--Sauer--Wegner] \label{conj:luecksauerwegner}
Let \(\Gamma\) and \(\Lambda\) be \(\det\)-\(L^2\)-acyclic groups.  Assume that \(\Gamma\) and \(\Lambda\) are measure equivalent of index \(c\).  Then \(\rho^{(2)}(\Gamma) = c \cdot \rho^{(2)}(\Lambda)\).
\end{conjecture}

This conjecture appears in \cite{Lueck-Sauer-Wegner:L2-torsion}*{Conjecture 1.2} where it is proven to hold true if measure equivalence is replaced by the way more rigorous notion of \emph{uniform measure equivalence}.  Regarding the original \fullref{conj:luecksauerwegner} the authors state that evidence comes from the similar formal behavior of Euler characteristic and \(L^2\)-torsion as well as from computations.  Our \fullref{thm:virtuall2torsion} together with a rigidity theorem due to Furman \cite{Furman:MERigidity} adds the following piece of evidence.

\begin{theorem} \label{thm:l2torsionmeasureequivalence}
Let \(\mathcal{L}^{\textup{even}}\) be the class of \(\det\)-\(L^2\)-acyclic groups that are measure equivalent to a lattice in a connected simple linear Lie group with even deficiency.  Then \textup{\fullref{conj:luecksauerwegner}} holds true for \(\mathcal{L}^{\textup{even}}\).
\end{theorem}   

Of course in fact \(\rho^{(2)}(\Gamma) = 0 \) for all \(\Gamma \in \mathcal{L}^{\textup{even}}\), which one might find unfortunate.  On the other hand, \(\mathcal{L}^{\textup{even}}\) contains various complete measure equivalence classes of \(\det\)-\(L^2\)-acyclic groups so that \fullref{thm:l2torsionmeasureequivalence} certainly has substance.  Gaboriau points out in \cite{Gaboriau:Examples}*{p.\,1810} that apart from amenable groups and lattices in connected simple linear Lie groups of higher rank, no more measure equivalence classes of groups have completely been understood so far.

In order to prove Theorems~\ref{thm:novikovshubinqrankone} and \ref{thm:virtuall2torsion} a close understanding of the Borel--Serre compactification is indispensable.  Thus we will give a detailed exposition in \fullref{sec:borelserre} of this paper.  We will closely follow the presentation in \cite{Borel-Ji:Compactifications}*{Chapter III.9} but unlike there, we include disconnected algebraic groups and give a sharpened version of \cite{Borel-Ji:Compactifications}*{Lemma~III.16.2, p.\,371} in \fullref{prop:closureofep} to stress the recursive character of the construction.  \fullref{sec:l2invariants} gives a brief introduction to \(L^2\)-invariants and their basic properties.  \fullref{sec:l2ofborelserre} forms the main part of this article where we compute \(L^2\)-invariants of the Borel--Serre compactification and conclude the results as presented in this introduction.  The strategy is to reduce the computation of \(L^2\)-invariants from the entire Borel--Serre bordification  \(\overline{X}\) to the boundary components \(e(\Pp)\).  If \(\Pp\) is minimal parabolic, then a certain subgroup of \(\Gamma\) acts cocompactly on \(e(\Pp)\) so that the results of Olbrich can be applied to the boundary symmetric space \(X_\Pp\) whereas the nilpotent factor \(N_P\) can be dealt with by results of M.\,Rumin and C.\,Wegner.  The material in this article is part of the author's doctoral thesis \cite{Kammeyer:L2-invariants}. It was written within the project ``\(L^2\)-invariants and quantum groups'' funded by the German Research Foundation (DFG).  I wish to thank my advisor Thomas Schick for many helpful suggestions.

\section{Borel-Serre compactification}
\label{sec:borelserre}

In this section we introduce the Borel--Serre compactification of a locally symmetric space mostly following the modern treatment by A.\,Borel and L.\,Ji \cite{Borel-Ji:Compactifications}*{Chapter III.9, p.\,326}.  The outline is as follows.  In \fullref{sec:arithmeticsubgroups} we recall basic notions of linear algebraic groups, their arithmetic subgroups and associated locally symmetric spaces.  \fullref{sec:rationalparabolicsubgroups} studies rational parabolic subgroups and their Langlands decompositions.  These induce horospherical decompositions of the symmetric space.  We classify rational parabolic subgroups up to conjugacy in terms of parabolic roots.  \fullref{sec:bordification} introduces and examines the bordification, a contractible manifold with corners which contains the symmetric space as an open dense set.  In \fullref{sec:quotients} we see that the group action extends cocompactly to the bordification.  The compact quotient gives the desired Borel--Serre compactification.  We will examine its constituents to some detail.

\subsection{Algebraic groups and arithmetic subgroups}
\label{sec:arithmeticsubgroups}

Let \(\Gg \subset \textup{GL}(n,\C)\) be a reductive linear algebraic group defined over \(\Q\) satisfying the following two conditions.
\begin{enumerate}[(I)]
\item \label{cond:anisotropiccenter} We have \(\chi^2 = 1\) for all \(\chi \in X_\Q(\Gg)\).
\item \label{cond:centralizermeetscomponents} The centralizer \(\mathcal{Z}_\Gg(\mathbf{T})\) of each maximal \(\Q\)-split torus \(\mathbf{T} \subset \Gg\) meets every connected component of \(\Gg\).
\end{enumerate}

This class of groups appears in \cite{Harish-Chandra:Automorphic}*{p.1}.  Condition (\ref{cond:anisotropiccenter}) implies that the group \(X_\Q(\Gg^0)\) of \(\Q\)-characters on the unit component of \(\Gg\) is trivial.  Thus \(\Gg\) has no central \(\Q\)-split torus.  Note that the structure theory of reductive algebraic groups is usually derived for connected groups, see for example \cite{Borel:LinearAlgebraicGroups}*{Chapter IV}.  But if one tries to enforce condition (\ref{cond:anisotropiccenter}) for a connected reductive \(\Q\)-group \(\Hh\) by going over to \(\bigcap_{\chi \in X_\Q(\Hh)} \ker \chi^2\), the resulting group will generally be disconnected.  That is why we impose the weaker condition (\ref{cond:centralizermeetscomponents}) which will turn out to be good enough for our purposes. 

A subgroup \(\Gamma \subset \Gg(\Q)\) is called \emph{arithmetic} if it is commensurable with \(\Gg(\Z)\).  This means \(\Gamma \cap \Gg(\Z)\) has finite index both in \(\Gamma\) and in \(\Gg(\Z)\).  The real points \(G = \Gg(\R)\) form a reductive Lie group with finitely many connected components \cite{Borel:LinearAlgebraicGroups}*{Section 24.6(c)(i), p.\,276}.  By a theorem of A.\,Borel and Harish-Chandra \cite{Borel-Harish-Chandra:Arithmetic}*{Theorem~9.4, p.\,522} condition (\ref{cond:anisotropiccenter}) implies that an arithmetic subgroup \(\Gamma \subset \Gg(\Q)\) is a \emph{lattice} in \(G\), which means the quotient space \(G/\Gamma\) has finite \(G\)-invariant measure.  Selberg's Lemma \cite{Alperin:SelbergsLemma} says that \(\Gamma\) has torsion-free subgroups of finite index.  We want to assume that \(\Gamma\) is torsion-free to begin with.  This ensures that \(\Gamma\) acts freely and properly from the left on the \emph{symmetric space} \(X = G / K\) where \(K\) is a maximal compact subgroup of \(G\).

Corresponding to \(K\) there is a \emph{Cartan involution} \(\theta_K\) on \(G\) which extends to an algebraic involution of \(\Gg\) \cite{Borel-Serre:Corners}*{Definition 1.7, p.\,444}.  If \(G\) is semisimple, \(\theta_K\) is the usual Cartan-involution.  The symmetric space \(X\) is connected because \(K\) meets every connected component of \(G\).  In general, it is the product of a symmetric space of noncompact type and a Euclidean factor.  The quotient \(\Gamma \backslash X = \Gamma \backslash G / K\) is called a \emph{locally symmetric space}.  The locally symmetric space \(\Gamma \backslash X\) is a connected finite-volume Riemannian manifold and in fact a \emph{classifying space} for \(\Gamma\) because its universal covering \(X\) is contractible.  The quotient \(\Gamma \backslash G\) or equivalently the locally symmetric space \(\Gamma \backslash X\) is compact if and only if \(\rank_\Q(\Gg) = 0\).

\subsection{Rational parabolic subgroups}
\label{sec:rationalparabolicsubgroups}

If \(\Gg\) is connected, a closed \(\Q\)-subgroup \(\Pp \subset \Gg\) is called a \emph{rational parabolic subgroup} if \(\Gg / \Pp\) is a complete (equivalently projective) variety.  If \(\Gg\) is not connected, we say that a closed \(\Q\)-subgroup \(\Pp \subset \Gg\) is a \emph{rational parabolic subgroup} if it is the normalizer of a rational parabolic subgroup of \(\Gg^0\).  These definitions are compatible because rational parabolic subgroups of connected groups are self-normalizing.  It is clear that \(\Pp^0 = \Pp \cap \Gg^0\), and condition (\ref{cond:centralizermeetscomponents}) on \(\Gg\) ensures that \(\Pp\) meets every connected component of \(\Gg\) \cite{Harish-Chandra:Automorphic}*{Lemma~1, p.\,2}, so \(\Gg / \Pp\) is complete.

Given a rational parabolic subgroup \(\Pp \subset \Gg\) we set \(\Nn_\Pp = \mathbf{R}_u(\Pp)\) and we denote by \(\Ll_\Pp = \Pp / \Nn_\Pp\) the \emph{Levi quotient} of \(\Pp\).  Let \(\Ss_\Pp \subset \Ll_\Pp\) be the maximal central \(\Q\)-split torus and set \(\Mm_\Pp = \bigcap_{\chi \in X_\Q(\Ll_\Pp)} \ker \chi^2\) where \(X_\Q(\Ll_\Pp)\) denotes the group of \(\Q\)-characters of \(\Ll_\Pp\).  The \(\Q\)-group \(\Mm_\Pp\) is reductive and satisfies conditions (\ref{cond:anisotropiccenter}) and (\ref{cond:centralizermeetscomponents}).  It complements \(\Ss_\Pp\) as an \emph{almost direct product}  in \(\Ll_\Pp\) \cite{Harish-Chandra:Automorphic}*{p.\,3}.  This means \(\Ll_\Pp = \Ss_\Pp \Mm_\Pp\) and \(\Ss_\Pp \cap \Mm_\Pp\) is finite.  For the groups of real points \(L_\Pp = \Ll_\Pp(\R)\), \(A_\Pp = \Ss_\Pp(\R)^0\) and \(M_\Pp = \Mm_\Pp(\R)\) the situation is even better behaved.  One can verify that \(L_\Pp = A_\Pp M_\Pp\) but now the finite group \(A_\Pp \cap M_\Pp\) is actually trivial because \(A_\Pp\) is torsion-free.  Since both \(A_\Pp\) and \(M_\Pp\) are normal, the product is direct.  We would like to lift these decompositions to some Levi \(k\)-subgroup of \(\Pp\).  The following result due to A.\,Borel and J.-P.\,Serre asserts that the maximal compact subgroup \(K \subset G\) singles out a canonical choice for doing so \cite{Borel-Serre:Corners}*{Proposition~1.8, p.\,444}.  The caveat is that \(k = \Q\) needs to be relaxed to \(k = \R\).  We view \(x_0 = K\) as a base point in the symmetric space \(X\).

\begin{proposition}
\label{prop:uniquelevisubgroup}
Let \(\Pp \subset \Gg\) be a rational parabolic subgroup and let \(K \subset G\) be maximal compact.  Then \(\Pp\) contains one and only one \(\R\)-Levi subgroup \(\Ll_{\Pp, x_0}\) which is stable under \(\theta_K\).
\end{proposition}

We remark that for a given \(\Pp\), the maximal compact subgroup \(K\) which is identified with the base point \(x_0 = K\) in \(X\) can always be chosen such that \(\Ll_{\Pp, x_0}\) is a \(\Q\)-group.  In fact, \(\Ll_{\mathbf{Q}, x_0}\) is then a \(\Q\)-group for all parabolic subgroups \(\mathbf{Q} \subset \Gg\) that contain \(\mathbf{P}\).  This follows from the proof of \cite{Borel-Ji:Compactifications}*{Proposition~III.1.11, p.\,273}.  In this case we will say that \(x_0\) is a \emph{rational base point} for \(\Pp\).  In general however, there is no universal base point \(x_0\) such that the \(\theta_K\)-stable Levi subgroups of all rational parabolic subgroups would be defined over \(\Q\) \cite{Goresky-Harder-MacPherson:Weighted}*{Section 3.9, p.\,151}.

The canonical projection \(\pi \co \Ll_{\Pp, x_0} \rightarrow \Ll_\Pp\) is an \(\R\)-isomorphism.  The groups \(\Ss_\Pp\) and \(\Mm_\Pp\) lift under \(\pi\) to the \(\R\)-subgroups \(\Ss_{\Pp, x_0}\) and \(\Mm_{\Pp, x_0}\) of \(\Pp\).  The rational parabolic subgroup \(\Pp\) thus has the decomposition
\begin{equation} \label{eqn:rationalparabolicdecomposition} \Pp = \Nn_\Pp \Ss_{\Pp, x_o} \Mm_{\Pp, x_0} \cong \Nn_\Pp \rtimes (\Ss_{\Pp, x_0} \Mm_{\Pp, x_0}) \end{equation}
where \(\Ll_{\Pp, x_0} = \Ss_{\Pp, x_0} \Mm_{\Pp, x_0}\) is an almost direct product.  Similarly the Lie groups \(L_\Pp\), \(A_\Pp\) and \(M_\Pp\) lift to the Lie subgroups \(L_{\Pp, x_0}\), \(A_{\Pp, x_0}\) and \(M_{\Pp, x_0}\) of the \emph{cuspidal group} \(P = \Pp(\R)\).

\begin{definition}
\label{def:rationallanglandsdecomposition} The point \(x_0 \in X\) yields the \emph{rational Langlands decomposition}
\[ P = N_P A_{\Pp, x_0} M_{\Pp, x_0} \cong N_P \rtimes (A_{\Pp, x_0} \times M_{\Pp, x_0}). \]
\end{definition}

We intentionally used a non-bold face index for \(N_P = \Nn_\Pp(\R)\) because \(N_P\) coincides with the unipotent radical of the linear Lie group \(P\).  The number \(\textup{s-rank}(\Pp) = \dim_\R A_{\Pp, x_0}\)  is called the \emph{split rank} of \(\Pp\).  Let \(K_P = P \cap K\) and \(K_P' = \pi(K_P)\).  Inspecting \cite{Borel-Serre:Corners}*{Proposition~1.8, p.\,444} we see that \(K_P \subset L_{\Pp, x_0}\) so \(K_P' \subset L_\Pp\).  Since \(K_P'\) is compact, we have \(\chi(K_P') \subset \{\pm1\}\) for each \(\chi \in X_\Q(\Ll_\Pp)\) so that actually \(K_P' \subset M_\Pp\) and thus \(K_P \subset M_{\Pp, x_0}\).  Moreover \(G = PK\) so that \(P\) acts transitively on the symmetric space \(X = G/K\).

\begin{definition}
\label{def:rationalhorosphericaldecomposition}
The map \((n, a, m K_P) \mapsto namK\) is a real analytic diffeomorphism
\[ N_P \times A_{\Pp, x_0} \times X_{\Pp, x_0} \cong X \]
of manifolds called the \emph{rational horospherical decomposition} of \(X\) with respect to \(\Pp\) and \(x_0\) and with \emph{boundary symmetric space} \(X_{\Pp, x_0} = M_{\Pp, x_0} / K_P\).
\end{definition}

Note that \(K_P \subset M_{\Pp, x_0}\) is maximal compact as it is even so in \(P\) \cite{Borel-Serre:Corners}*{Proposition~1.5, p.\,442}.  Write an element \(p \in P\) according to the rational Langlands decomposition as \(p = nam\) and write a point \(x_1 \in X\) according to the rational horospherical decomposition as \(x_1 = (n_1, a_1, m_1 K_P)\).  Then we see that the left-action of \(P\) on \(X\) is given by
\[ nam.(n_1, a_1, m_1 K_P) = (n\ {}^{am} n_1 , a a_1, m m_1 K_P), \]
where we adopt the convention to write \({}^h g\) for the conjugation \(h g h^{-1}\). 

The horospherical decomposition realizes the symmetric space \(X\) as the product of a nilmanifold, a flat manifold and yet another symmetric space \(X_{\Pp, x_0}\).  The isomorphism \(\pi\) identifies the latter one with the symmetric space \(X_\Pp = M_\Pp / K_P'\).  It is the symmetric space of the reductive \(\Q\)-group \(\Mm_\Pp\) which meets conditions (\ref{cond:anisotropiccenter}) and (\ref{cond:centralizermeetscomponents}).  The group \(\Mm_\Pp\) inherits the arithmetic lattice \(\Gamma_{M_\Pp}'\) \label{page:gammamp} which is the image of \(\Gamma_P = \Gamma \cap \mathcal{N}_G(P)\) under the projection \(P \rightarrow P/N_P \cong L_P\).  Here we have \(\Gamma_{M_\Pp}' \subset M_\Pp\) because \(\chi(\Gamma_{M_\Pp}') \subset \{ \pm 1 \}\) for all \(\chi \in X_\Q(\Ll_\Pp)\) as \(\chi(\Gamma_{M_\Pp}') \subset \mathbf{GL}(1,\Q)\) is arithmetic.  In general \(\Gamma_{M_\Pp}'\) might have torsion elements.  But there is a condition on \(\Gamma\) that ensures it does not.

\begin{definition} \label{def:neat}
A matrix \(g \in \textup{GL}(n,\Q)\) is called \emph{neat} if the subgroup of \(\C^*\) generated by the eigenvalues of \(g\) is torsion-free.  A subgroup of \(\textup{GL}(n, \Q)\) is called neat if all of its elements are neat.
\end{definition}

The notion of neatness is due to J.-P.\,Serre. It appears first in \cite{Borel:IntroductionAuxGroupes}*{Section 17.1, p.\,117}.  A neat subgroup is obviously torsion-free.  Every arithmetic subgroup of a linear algebraic \(\Q\)-group has a neat subgroup of finite index \cite{Borel:IntroductionAuxGroupes}*{Proposition~17.4, p.\,118} and neatness is preserved under morphisms of linear algebraic groups \cite{Borel:IntroductionAuxGroupes}*{Corollaire 17.3, p.\,118}.  Therefore \(\Gamma_{M_\Pp}'\) is neat if \(\Gamma\) is, and in that case \(\Gamma_{M_\Pp}'\) acts freely and properly on the boundary symmetric space \(X_\Pp\).  We observe that \(\rank_\Q(\Mm_\Pp) = \rank_\Q(\Gg) - \dim A_\Pp\).  In this sense the locally symmetric space \(\Gamma_{M_\Pp}' \backslash X_\Pp\) is closer to being compact than the original \(\Gamma \backslash X\).  This is a key observation for the construction of the Borel--Serre compactification.  If in particular \(\Pp\) is a \emph{minimal} rational parabolic subgroup, then \(\Ss_{\Pp, x_0} \subset \Pp\) is \(G\)-conjugate to a maximal \(\Q\)-split torus of \(\Gg\) so that \(\rank_\Q(\Mm_\Pp) = 0\) and thus \(\Gamma_{M_\Pp}' \backslash X_\Pp\) is compact.

Now the group \(\Mm_\Pp\) has itself rational parabolic subgroups \(\Qq'\) whose cuspidal subgroups \(Q'\) have a Langlands decomposition \(Q' = N_{Q'} A_{\Qq'\!, x_0'} M_{\Qq'\!, x_0'}\) with respect to the base point \(x_0' = K_P'\).  The isomorphism \(\pi\) identifies those groups as subgroups of \(M_{\Pp, x_0}\).  We set \(N_Q^* = N_P N_{Q'} \cong N_P \rtimes N_{Q'}\), \(A_{\Qq, x_0}^* = A_{\Pp, x_0} A_{\Qq'\!, x_0'} = A_{\Pp, x_o} \rtimes A_{\Qq'\!, x_0'}\) and \(M_{\Qq, x_0}^* = M_{\Qq'\!, x_0'}\).  Then we define \(Q^* = N_Q^* A_{\Qq, x_0}^* M_{\Qq, x_0}^*\).  The group \(Q^*\) is the cuspidal group of a rational parabolic subgroup \(\Qq^*\) of \(\Gg\) such that \(\Qq^* \subset \Pp\).  Equivalently, \(\Qq^*\) is a rational parabolic subgroup of \(\Pp\).  The Langlands decomposition of \(Q^*\) with respect to \(x_0\) is the decomposition given in its construction.

\begin{lemma}
\label{lemma:paraboliccorrespondence}
The map \(\Qq' \mapsto \Qq^*\) gives a bijection of the set of rational parabolic subgroups of \(\Mm_\Pp\) to the set of rational parabolic subgroups of \(\Gg\) contained in \(\Pp\).
\end{lemma}

This is \cite{Harish-Chandra:Automorphic}*{Lemma~2, p.\,4}.  We use the inverse of this correspondence to conclude that for every rational parabolic subgroup \(\Qq = \Qq^* \subset \Pp\) we obtain a \emph{rational horospherical decomposition of the boundary symmetric space}
\begin{equation} \label{eqn:horosphericalofboundary} X_{\Pp, x_0} \cong X_{\Pp} \cong N_{Q'} \times A_{\Qq'\!, x_0'} \times X_{\Qq'\!, x_0'}. \end{equation}
In the case \(\Pp = \Gg\) condition (\ref{cond:anisotropiccenter}) gives \(\Mm_{\Gg, x_0} = \Gg\) so that we get back the original rational horospherical decomposition of \fullref{def:rationalhorosphericaldecomposition}.

In the rest of this section we will recall the classification of rational parabolic subgroups of \(\Gg\) up to conjugation in \(\Gg(\Q)\) in terms of parabolic roots \cite{Harish-Chandra:Automorphic}*{Chapter 1, pp.\,3--4}.  Let \(\fr{g}^0\), \(\fr{p}\), \(\fr{n}_P\), \(\fr{a}_{\Pp, x_0}\) and \(\fr{m}_{\Pp, x_0}\) be the Lie algebras of the Lie groups \(G\), \(P\), \(N_P\), \(A_{\Pp, x_0}\) and \(M_{\Pp, x_0}\).  From the viewpoint of algebraic groups, these Lie algebras are given by \(\R\)-linear left-invariant derivations of the field of rational functions defined over \(\R\) on the unit components of \(\Gg\), \(\Pp\), \(\Nn_\Pp\), \(\Ss_{\Pp, x_0}\) and \(\Mm_{\Pp, x_0}\), respectively.  A linear functional \(\alpha\) on \(\fr{a}_{\Pp, x_0}\) is called a \emph{parabolic root} if the subspace
\[ \fr{n}_{P, \alpha} = \{ n \in \fr{n}_P \co \ad(a)(n) = \alpha(a)n \text{ for all } a \in \fr{a}_{\Pp, x_0} \} \]
of \(\fr{n}_P\) is nonzero.  We denote the set of all parabolic roots by \(\Phi(\fr{p}, \fr{a}_{\Pp, x_0})\).  If \(l = \dim \fr{a}_{\Pp, x_0}\), there is a unique subset \(\Delta(\fr{p}, \fr{a}_{\Pp, x_0}) \subset \Phi(\fr{p}, \fr{a}_{\Pp, x_0})\) of \(l\) \emph{simple parabolic roots} such that every parabolic root is a unique linear combination of simple ones with nonnegative integer coefficients.  The group \(A_{\Pp, x_0}\) is exponential so that \(\exp \co \fr{a}_{\Pp, x_0} \rightarrow A_{\Pp, x_0}\) is a diffeomorphism with inverse ``\(\log\)''.  Therefore we can evaluate a parabolic root \(\alpha \in \Phi(\fr{p}, \fr{a}_{\Pp, x_0})\) on elements \(a \in A_{\Pp, x_0}\) setting \(a^\alpha = \exp (\alpha(\log a))\) where now ``\(\exp\)'' is the ordinary real exponential function.

\label{page:classifyparabolic} The subsets of \(\Delta(\fr{p}, \fr{a}_{\Pp, x_0})\) classify the rational parabolic subgroups of \(\Gg\) that contain \(\Pp\) as we will now explain.  Let \(I \subset \Delta(\fr{p},\fr{a}_{\Pp, x_0})\) be a subset and let \(\Phi_I \subset \Phi(\fr{p}, \fr{a}_{\Pp, x_0})\) be the set of all parabolic roots that are linear combinations of simple roots in \(I\).  Set \(\fr{a}_I = \bigcap_{\alpha \in I} \ker \alpha\) and \(\fr{n}_I = \bigoplus_{\alpha \in \Sigma} \fr{n}_{P, \alpha}\) where \(\Sigma = \Sigma(\fr{p}, \fr{a}_{\Pp, x_0})\) denotes the set of all parabolic roots which do not lie in \(\Phi_I\).  Consider the sum \(\fr{p}_I = \fr{n}_I \oplus \fr{z}(\fr{a}_I)\) of \(\fr{n}_I\) and the centralizer of \(\fr{a}_I\) in \(\fr{g}^0\).  Let \(P_I = \mathcal{N}_G(\fr{p}_I)\) be the normalizer of \(\fr{p}_I\) in \(G\).  If \(x_1 \in X\) is a different base point, then \(x_1 = p.x_0\) for some \(p \in P\) and \(\fr{a}_{\Pp, x_1} = {}^p \fr{a}_{\Pp, x_0}\) as well as \(\fr{n}_{(I^p)} = {}^p\fr{n}_I\).  It follows that the group \(P_I\), thus its Zariski closure \(\Pp_I\), is independent of the choice of base point.  Since rational base points exist for \(\Pp\), the Lie algebra of \(\Pp_I\), which as a variety is given by \(\C\)-linear left-invariant derivations of the field of rational functions on \(\Pp_I^0\), is defined over \(\Q\).  It follows that \(\Pp_I\) is a \(\Q\)-group \cite{Harish-Chandra:Automorphic}*{p.\,1}.  In fact, \(\Pp_I\) is a rational parabolic subgroup of \(\Gg\) with cuspidal group \(P_I\).  Let \(N_I\) and \(A_I\) be the Lie subgroups of \(P_I\) with Lie algebras \(\fr{n}_I\) and \(\fr{a}_I\).  Then \(N_I \subset P_I\) is the unipotent radical and \(A_I = \Ss_{\Pp_{\!I}\!, x_0}(\R)^0\).  The parabolic roots \(\Phi(\fr{p}_I, \fr{a}_I)\) are the restrictions of \(\Sigma(\fr{p}, \fr{a}_{\Pp, x_0})\) to \(\fr{a}_I\) where simple parabolic roots restrict to the simple ones \(\Delta(\fr{p}_I, \fr{a}_I)\) of \(\fr{p}_I\).

Every rational parabolic subgroup of \(\Gg\) that contains \(\Pp\) is of the form \(\Pp_I\) for a unique \(I \subset \Delta(\fr{p}, \fr{a}_{\Pp, x_0})\).  The two extreme cases are \(\Pp_\emptyset = \Pp\) and \(\Pp_{\Delta(\fr{p}, \fr{a}_{\Pp, x_0})} = \Gg\).  If \(\Pp\) is minimal, the groups \(\Pp_I\) form a choice of so called \emph{standard rational parabolic subgroups}.  Every rational parabolic subgroup of \(\Gg\) is \(\Gg(\Q)\)-conjugate to a unique standard one.  Whence there are only finitely many rational parabolic subgroups up to conjugation in \(\Gg(\Q)\).  There are even only finitely many when we restrict ourselves to conjugating by elements of an arithmetic subgroup \(\Gamma \subset \Gg(\Q)\).  This is clear from the following result of A.\,Borel \cite{Harish-Chandra:Automorphic}*{p.\,5}.

\begin{proposition}
\label{prop:finitelymanygammaparabolic}
Let \(\Pp \subset \Gg\) be a rational parabolic subgroup and let \(\Gamma \subset \Gg(\Q)\) be an arithmetic subgroup. Then the set \(\Gamma \backslash \,\Gg(\Q)\, / \Pp(\Q)\) is finite.
\end{proposition}

\subsection{Bordification}
\label{sec:bordification}

From now on we drop \(x_0\) and \(x_0'\) from our notation.  The resulting notational collisions \(A_\Pp = A_{\Pp, x_0}\), \(M_\Pp = M_{\Pp, x_o}\) and \(X_\Pp = X_{\Pp, x_0}\) regarding Levi quotients and Levi subgroups are justified by \fullref{prop:uniquelevisubgroup} and the discussion throughout the preceding section.  We will use the symbol ``\(\textstyle\bigcupdot\)'' for general disjoint unions in topological spaces, whereas the symbol ``\(\textstyle\coprod\)'' is reserved for the true categorical coproduct.

Let \(\Pp \subset \Gg\) be a rational parabolic subgroup.  It determines the rational horospherical decomposition \(X = N_P \times A_\Pp \times X_\Pp\) of \fullref{def:rationalhorosphericaldecomposition}.  Define the \emph{boundary component} of \(\Pp\) by \(e(\Pp) = N_P \times X_\Pp\).  Then as a set, the \emph{Borel--Serre bordification} \(\overline{X}\) of the symmetric space \(X\) is given by the countable disjoint union
\[ \overline{X} = \textstyle\coprod\limits_{\Pp \subset \Gg} e(\Pp) \]
of all boundary components of rational parabolic subgroups \(\Pp \subset \Gg\).  This includes the symmetric space \(X = e(\Gg)\).  In order to topologize the set \(\overline{X}\) we introduce different coordinates on \(e(\Pp)\) for every parabolic subgroup \(\Qq \subset \Pp\).  We do so by writing the second factor in \(e(\Pp) = N_P \times X_\Pp\) according to the rational horospherical decomposition of the boundary symmetric space \(X_\Pp = N_{Q'} \times A_{\Qq'} \times X_{\Qq'}\) given in \eqref{eqn:horosphericalofboundary}.  From the preparation of \fullref{lemma:paraboliccorrespondence} we get \(N_Q = N_P N_{Q'}\) and \(M_Q = M_{Q'}\) so that we are left with
\begin{equation} \label{eqn:coordinatesofep}  e(\Pp) = N_Q \times A_{\Qq'} \times X_{\Qq}. \end{equation}
The closed sets of \(\overline{X}\) are now determined by the following \emph{convergence class of sequences} \cite{Borel-Ji:Compactifications}*{I.8.9--I.8.13, p.\,113}.

A sequence \((x_i)\) of points in \(e(\Pp)\) converges to a point \(x \in e(\Qq)\) if \(\Qq \subset \Pp\) and if for the coordinates \(x_i = (n_i, a_i, y_i)\) of \eqref{eqn:coordinatesofep} and \(x = (n,y)\) of \(e(\Qq) = N_Q \times X_\Qq\) the following three conditions hold true.
\begin{enumerate}[(i)]
\item \label{item:rootstoinfinity} \(a_i^\alpha \rightarrow +\infty\) for each \(\alpha \in \Phi(\fr{q}', \fr{a}_{\Qq'})\),
\item \label{item:nilpotentconvergence} \(n_i \rightarrow n\) within \(N_Q\),
\item \label{item:boundaryconvergence} \(y_i \rightarrow y\) within \(X_\Qq\).
\end{enumerate}
 A general sequence \((x_i)\) of points in \(\overline{X}\) converges to a point \(x \in e(\Qq)\) if for each \(\Pp \subset \Gg\) every infinite subsequence of \((x_i)\) within \(e(\Pp)\) converges to \(x\).

Note that in the case  \(\Qq = \Pp\) the set \(\Phi(\fr{q}', \fr{a}_{\Qq'})\) is empty so that condition (\ref{item:rootstoinfinity}) is vacuous.  We therefore obtain the convergence of the natural topology of \(e(\Pp)\).  In particular, the case \(\Qq = \Pp = \Gg\) gives back the natural topology of \(X\).  It is clear that we obtain the same set \(\overline{X}\) with the same class of sequences if we go over from \(\Gg\) to \(\Gg^0\).  We thus may cite \cite{Borel-Ji:Compactifications}*{Section III.9.2, p.\,328} where it is stated that this class of sequences does indeed form a convergence class of sequences.  This defines the topology of \(\overline{X}\).

Since a sequence \((x_i)\) in \(e(\Pp)\) can only converge to a point \(x \in e(\Qq)\) if \(\Qq \subset \Pp\), it is immediate that the \emph{Borel--Serre boundary} \(\partial \overline{X} \subset \overline{X}\) of \(\overline{X}\) defined as
\begin{equation} \label{eqn:borelserreboundary}
\partial \overline{X} = \textstyle\bigcupdot\limits_{\Pp \subsetneq \Gg} e(\Pp)
\end{equation}
is closed in \(\overline{X}\).  Whence its complement \(e(\Gg) = X \subset \overline{X}\) is open.  The following proposition sharpens \cite{Borel-Ji:Compactifications}*{Lemma~III.16.2, p.\,371}.

\begin{proposition}
\label{prop:closureofep}
The closure of the boundary component \(e(\Pp)\) in the bordification \(\overline{X}\) can be canonically identified with the product
\[ \overline{e(\Pp)} = N_P \times \overline{X}_\Pp \]
where \(\overline{X}_\Pp\) is the Borel--Serre bordification of the boundary symmetric space \(X_\Pp\).
\end{proposition}

\begin{proof}
By construction of the convergence class of sequences we have
\begin{equation} \label{eqn:closureofep} \overline{e(\Pp)} = \textstyle\bigcupdot\limits_{\Qq \subset \Pp} e(\Qq). \end{equation}
In terms of the rational parabolic subgroup \(\Qq' \subset \Mm_\Pp\) of \fullref{lemma:paraboliccorrespondence} the boundary component \(e(\Qq)\) can be expressed as
\begin{equation} \label{eqn:relativeboundarycomponent} e(\Qq) = N_Q \times X_\Qq = N_P \times N_{Q'} \times X_{\Qq'} = N_P \times e(\Qq'). \end{equation}
In the distributive category of sets we thus obtain 
\[ \overline{e(\Pp)} = \textstyle\coprod\limits_{\Qq \subset \Pp} e(\Qq) = \!\textstyle\coprod\limits_{\Qq' \subset \Mm_\Pp}\! N_P \times e(\Qq') = N_P \times \!\!\textstyle\coprod\limits_{\Qq' \subset \Mm_\Pp}\!\!\! e(\Qq') = N_P \times \overline{X}_\Pp. \]
We have to verify that this identifies the spaces \(\overline{e(\Pp)}\) and \(N_P \times \overline{X}_\Pp\) also topologically if we assign the bordification topology to \(\overline{X}_\Pp\).  For this purpose we show that the natural convergence classes of sequences on \(\overline{e(\Pp)}\) and \(N_P \times \overline{X}_\Pp\) coincide.  Let us refine our notation and write \(\Qq' = \Qq | \Pp\) to stress that \(\Qq' \subset \Mm_\Pp\).  Let \(\mathbf{R} \subset \Qq\) be a third rational parabolic subgroup.  Then the equality \(M_\Qq = M_{\Qq | \Pp}\) implies the cancellation rule \(\mathbf{R} | \Qq = (\mathbf{R} | \Pp) | (\Qq | \Pp)\).  Incorporating coordinates for \(e(\Qq)\) with respect to \(\mathbf{R}\) as in \eqref{eqn:coordinatesofep}, equation \eqref{eqn:relativeboundarycomponent} can now be written as
\[ e(\Qq) = N_R \times A_{\mathbf{R}|\Qq} \times X_\mathbf{R} = N_P \times (N_{R | P} \times A_{(\mathbf{R} | \Pp) | (\Qq | \Pp)} \times X_{\mathbf{R} | \Pp}). \]
Here the product \(N_{R | P} \times A_{(\mathbf{R} | \Pp) | (\Qq | \Pp)} \times X_{\mathbf{R} | \Pp}\) gives the coordinates \eqref{eqn:coordinatesofep} for \(e(\Qq | \Pp)\) with respect to \(\mathbf{R} | \Pp\).  Let \((n_i,a_i,y_i)\) be a sequence in \(e(\Qq)\) converging to \((n,y) \in e(\mathbf{R})\).  We decompose uniquely \(n_i = n_i^P n_i^{R | P}\) and \(n = n^P n^{R | P}\) according to \(N_R = N_P N_{R | P} \cong N_P \rtimes N_{R | P}\).  Then firstly \(n_i^P \rightarrow n^P\) in \(N_P\).  Secondly \((n_i^{R | P}, a_i, y_i)\) is a sequence in \(e(\Qq | \Pp)\) that converges to \((n^{R | P}, y) \in e(\mathbf{R} | \Pp)\) according to the convergence class of the bordification \(\overline{X}_\Pp\).  Since the convergence class of \(N_P \times \overline{X}_\Pp\) consists of the memberwise products of convergent sequences in \(N_P\) and the sequences in the convergence class of \(\overline{X}_\Pp\), this clearly proves the assertion.
\end{proof}

One special case of this proposition is \(\overline{e(\Gg)} = \overline{X}\).  The other important special case occurs when \(\Pp\) is a minimal rational parabolic subgroup.  Then \(\rank_\Q(\Mm_\Pp) = 0\) so that \(\overline{X}_\Pp = X_\Pp\) which means that \(e(\Pp)\) is closed.

As we have \(\overline{e(\Pp)} = \textstyle\bigcupdot e(\Qq)\), the union running over all \(\Qq \subset \Pp\), we should also examine the subset
\[ \underline{e(\Pp)} = \textstyle\bigcupdot\limits_{\Qq \supset \Pp} e(\Qq) \subset \overline{X}. \]
To this end consider the rational horospherical decomposition \(X = N_P \times A_\Pp \times X_\Pp\) of \(X\) given \(\Pp\).  Let \(\Delta(\fr{p}, \fr{a}_\Pp) = \{ \alpha_1, \ldots, \alpha_l \}\) be a numbering of the simple parabolic roots.  The map  \(a \mapsto (a^{-\alpha_1}, \ldots, a^{-\alpha_l})\) defines a coordinate chart \(\varphi_\Pp \co A_\Pp \rightarrow (\R_{>0})^l\).  The minus signs make sure the ``point at infinity'' of \(A_\Pp\) will correspond to the origin in \(\R^l\).  Let \(\overline{A}_\Pp\) be the closure of \(A_\Pp\) in \(\R^l\) under the embedding \(\varphi_\Pp\).  Given \(\Qq \supset \Pp\), let \(I \subset \Delta = \Delta(\fr{p},\fr{a}_\Pp)\) be such that \(\Qq = \Pp_I\) and set
\[ A_{\Pp, \Qq} = \exp ( \textstyle\bigcap\limits_{\alpha \in \Delta \setminus I} \ker \alpha ) \]%\{ \exp{a} \co a \in \ker \alpha \text{ for each } \alpha \in \Delta(\fr{p}, \fr{a}_\Pp) \setminus I \}. \]
Since the simple roots \(\Delta(\fr{p}, \fr{a}_\Pp)\) restrict to the simple roots \(\Delta(\fr{p}_I, \fr{a}_I)\), we obtain inclusions \(A_{\Pp, \Qq} \times \overline{A}_\Qq \subset \overline{A}_\Pp\).  If \(o_\Qq \in \overline{A}_\Qq\) denotes the origin, these inclusions combine to give a disjoint decomposition
\[ \overline{A}_\Pp = \textstyle\bigcupdot\limits_{\Qq \supset \Pp} A_{\Pp, \Qq} \times o_\Qq \]
of the corner \(\overline{A}_\Pp\) into the corner point (for \(\Qq = \Pp\)), the boundary edges, the boundary faces, \(\ldots\ \), the boundary hyperfaces and the interior (for \(\Qq = \Gg\)).  In the coordinates \(e(\Qq) = N_P \times A_{\Pp'} \times X_\Pp\) as in \eqref{eqn:coordinatesofep}, the group \(A_{\Pp'}\) can be identified with the group \(A_{\Pp, \Qq}\) \cite{Borel-Ji:Compactifications}*{Lemma~III.9.7, p.\,330}. It follows that the subset \(N_P \times A_{\Pp, \Qq} \times o_\Qq \times X_\Pp\) in \(N_P \times \overline{A}_\Pp \times X_\Pp\) can be identified with \(e(\Qq)\) and hence
\begin{equation} \label{eqn:corner} \underline{e(\Pp)} \cong N_P \times \overline{A}_\Pp \times X_\Pp \end{equation}
has the structure of a real analytic manifold with corners.  For a proof that the involved topologies match, we refer to \cite{Borel-Ji:Compactifications}*{Lemmas III.9.8--10, pp.\,330--332}.  The manifold \(\underline{e(\Pp)}\) is called the \emph{corner} in \(\overline{X}\) corresponding to the rational parabolic subgroup \(\Pp\).  The corners \(\underline{e(\Pp)}\) are open.  With their help neighborhood bases of boundary points in \(\overline{X}\) can be described \cite{Borel-Ji:Compactifications}*{Lemma~III.9.13, p.\,332}.  These demonstrate that \(\overline{X}\) is a Hausdorff space \cite{Borel-Ji:Compactifications}*{Proposition~III.9.14, p.\,333}.  The corners \(\underline{e(\Pp)}\) form an open cover of the bordification \(\overline{X}\).  One verifies that their analytic structures are compatible to conclude the following result \cite{Borel-Ji:Compactifications}*{Proposition~III.9.16, p.\,335}.

\begin{proposition}
\label{prop:realanalytic}
The bordification \(\overline{X}\) has a canonical structure of a real analytic manifold with corners.
\end{proposition}

If one wishes, the corners of \(\overline{X}\) can be smoothed to endow \(\overline{X}\) with the structure of a smooth manifold with boundary \cite{Borel-Serre:Corners}*{Appendix}.  The collar neighborhood theorem thus implies that \(\overline{X}\) is homotopy equivalent to its interior. 

\begin{corollary}
\label{cor:bordificationcontractible}
The bordification \(\overline{X}\) is contractible.
\end{corollary}

Another corollary of \fullref{prop:realanalytic} together with \fullref{prop:closureofep} is that the closures of boundary components \(\overline{e(\Pp)}\) are real analytic manifolds with corners as well.  In fact, the inclusion \(\overline{e(\Pp)} \subset \overline{X}\) realizes \(\overline{e(\Pp)}\) as a \emph{submanifold with corners} of \(\overline{X}\).  Note that topologically a manifold with corners is just a manifold with boundary.  We conclude this section with the observation that
\begin{equation}
\label{eqn:intersectionofboundarycomponents} \overline{e(\Pp)} \cap \overline{e(\Qq)} = \overline{e(\Pp \cap \Qq)}
\end{equation}
if \(\Pp \cap \Qq\) is rational parabolic.  Otherwise the intersection is empty.  Dually,
\[ \underline{e(\Pp)} \cap \underline{e(\Qq)} = \underline{e(\mathbf{R})} \]
where now \(\mathbf{R}\) denotes the smallest rational parabolic subgroup of \(\Gg\) that contains both \(\Pp\) and \(\Qq\).  If \(\mathbf{R} = \Gg\), the intersection equals \(X\).

\subsection{Quotients}
\label{sec:quotients}

We extend the action of \(\Gg(\Q)\) on \(X\) to an action on \(\overline{X}\).  Given \(g \in \Gg(\Q)\) and a rational parabolic subgroup \(\Pp\), let \(k \in K\), \(n \in N_P\), \(a \in A_\Pp\) and \(m \in M_\Pp\) such that \(g = kman\).  Note that we have swapped \(m\) and \(n\) compared to the order in the rational Langlands decomposition in \fullref{def:rationallanglandsdecomposition}.  This ensures that \(a\) and \(n\) are unique.  In contrast, the elements \(k\) and \(m\) can be altered from right and left by mutually inverse elements in \(K_P\).  Their product \(km\) is however well-defined.  We therefore obtain a well-defined map \(g. \co e(\Pp) \rightarrow e({}^k \Pp)\) setting
\begin{equation} \label{eqn:actiononbordification} g.(n_0, m_0K_P) = ({}^{kma} (n n_0), {}^k (m m_0) K_{{}^k\! P}). \end{equation}
Using the convergence class of sequences, one checks easily that this defines a continuous and in fact a real analytic action of \(\Gg(\Q)\) on \(\overline{X}\) which extends the action on \(X\) \cite{Borel-Ji:Compactifications}*{Propositions III.9.15--16, pp.\,333--335}.  The restricted action of \(\Gamma \subset \Gg(\Q)\) is proper \cite{Borel-Ji:Compactifications}*{Proposition~III.9.17, p.\,336} and thus free because \(\Gamma\) is torsion-free.  The quotient \(\Gamma \backslash \overline{X}\) is therefore Hausdorff and in fact a real analytic manifold with corners.  It is called the \emph{Borel--Serre compactification} of the locally symmetric space \(\Gamma \backslash X\) in view of the following result \cite{Borel-Ji:Compactifications}*{Theorem~III.9.18, p.\,337}.  

\begin{theorem} \label{thm:borelserrecompact}
The real analytic manifold with corners \(\Gamma \backslash \overline{X}\) is compact.
\end{theorem}

By \fullref{cor:bordificationcontractible} the Borel--Serre compactification \(\Gamma \backslash \overline{X}\) is a classifying space for \(\Gamma\).  The subgroup \(\Gamma_P = \Gamma \cap \mathcal{N}_G (P)\) of \(\Gamma\) leaves \(e(\Pp)\) invariant.  Let us denote the quotient by \(e'(\Pp) = \Gamma_P \backslash e(\Pp)\).  Since \(g.e(\Pp) \,\cap\, e(\Pp) = \emptyset\) for every \(g \in \Gamma\) that does  not lie in \(\Gamma_P\), we have the following disjoint decomposition of the quotient \(\Gamma \backslash \overline{X}\) \cite{Borel-Ji:Compactifications}*{Proposition~III.9.20, p.\,337}.

\begin{proposition}
\label{prop:decompofborelserre}
Let \(\Pp_1, \ldots, \Pp_r\) be a system of representatives of \(\Gamma\)-conjugacy classes of rational parabolic subgroups in \(\Gg\).  Then
\[ \Gamma \backslash \overline{X} = \textstyle\bigcupdot\limits_{i = 1}^r e'(\Pp_i). \]
\end{proposition}

The closure of \(e'(\Pp)\) in \(\Gamma \backslash \overline{X}\) is compact and has the decomposition
\begin{equation} \label{closureofeprimep} \overline{e'(\Pp)} = \textstyle\bigcupdot\limits_{\Qq \subset \Pp} e'(\Qq). \end{equation}
This follows from the compatibilities \(e'(\Pp) = \nu(e(\Pp))\) and \(\overline{e'(\Pp)} = \nu(\overline{e(\Pp)})\) and from \eqref{eqn:closureofep} where \(\nu \co \overline{X} \rightarrow \Gamma \backslash \overline{X}\) denotes the canonical projection \cite{Borel-Serre:Corners}*{Proposition~9.4, p.\,476}.  By \eqref{eqn:closureofep} and the remarks preceding \fullref{prop:decompofborelserre} we see that \(\overline{e'(\Pp)} = \nu(\overline{e(\Pp)})\) also equals \(\Gamma_P \backslash \overline{e(\Pp)}\).  We will examine this latter quotient.

Let \(\Gamma_{N_P} = \Gamma \cap N_P\).  The rational Langlands decomposition \ref{def:rationallanglandsdecomposition} defines a projection \(P \rightarrow M_\Pp\).  Let \(\Gamma_{M_\Pp}\) be the image of \(\Gamma_P\) under this projection.  Equivalently, \(\Gamma_{M_\Pp}\) is the canonical lifting under \(\pi\) of the group \(\Gamma_{M_\Pp}'\) defined below \fullref{def:rationalhorosphericaldecomposition}, see \cite{Borel-Ji:CompactificationsDiffGeo}*{Proposition~2.6, p.\,272}.  We should however not conceal a word of warning.  The lift \(\Gamma_{M_\Pp}' \rightarrow \Gamma_{M_\Pp}\) does not necessarily split the exact sequence
\[ 1 \longrightarrow \Gamma_{N_P} \longrightarrow \Gamma_P \longrightarrow \Gamma_{M_\Pp}' \longrightarrow 1, \]
not even if the suppressed base point was rational for \(\Pp\).  By \cite{Borel-Ji:CompactificationsDiffGeo}*{Propositions 2.6 and 2.8, p.\,272} we have \(\Gamma_P \subset N_P \Gamma_{M_\Pp} = N_P \Gamma_P\).  We analyze how the action of \(\Gamma_P\) on \(\overline{e(\Pp)}\) behaves regarding the decomposition \(\overline{e(\Pp)} = N_P \times \overline{X}_\Pp\) of \fullref{prop:closureofep}. 

\begin{proposition}
\label{prop:actiononoverlineep}
Let \(p \in \Gamma_P\) and let \(p = mn\) be its unique decomposition with \(m \in \Gamma_{M_\Pp}\) and \(n \in N_P\).  Let \((n_0, x) \in N_P \times \overline{X}_\Pp = \overline{e(\Pp)}\).  Then
\[ p.(n_0, x) = ({}^m(n n_0), m.x). \]
\end{proposition}

\begin{proof}
There is a unique rational parabolic subgroup \(\Qq \subset \Pp\) and there are unique elements \(n_0' \in N_{Q'}\) and \(m_0' \in M_{Q'}\) such that
\[ x = (n_0', m_0'K_{Q'}) \in N_{Q'} \times X_{\Qq'} = e(\Qq') \subset \overline{X}_\Pp. \]
We decompose \(m \in M_\Pp\) as \(m = km'a'n'\) with \(k \in K_P\), \(m' \in M_{\Qq'}\), \(a' \in A_{\Qq'}\) and \(n' \in N_{Q'}\).  By \eqref{eqn:relativeboundarycomponent} we have \(N_P \times e(\Qq') = e(\Qq) = N_Q \times X_\Qq\) and under this identification our element \((n_0, x)\) corresponds to \((n_0 n_0', m_0'K_Q)\).  We have \(p = km'a'(n'n)\) with \(m' \in M_{\Qq'} = M_\Qq\), \(a' \in A_{\Qq'} \subset A_\Qq\) and \(n'n \in N_Q\).  According to \eqref{eqn:actiononbordification} the element \(p\) therefore acts as
\[ p.(n_0 n_0', m_0' K_Q) = ({}^{km'a'}(n'nn_0n_0'), {}^k (m'm_0')K_{{}^k\! Q}). \]
For the left-hand factor we compute
\begin{align*}
{}^{km'a'}(n'nn_0n_0') &= {}^{km'a'}({}^{n'}(nn_0)n'n_0') = {}^{km'a'n'}\!(nn_0)\ {}^{km'a'}\!(n'n_0') = \\
&= {}^{m}\!(nn_0)\ {}^{km'a'}\!(n'n_0').
\end{align*}
Transforming back from \(N_Q \times X_\Qq\) to \(N_P \times e(\Qq')\) we therefore obtain
\[ p.(n_0,x) = ({}^{m}(nn_0), ({}^{km'a'}(n'n_0'), {}^k(m'm_0') K_{{}^k\! Q})) = ({}^{m}(nn_0), m.x). \proved \]
\end{proof}

If \(\Gamma\) is neat, then \fullref{prop:actiononoverlineep} makes explicit that we have a commutative diagram
\[ \begin{xy}
\xymatrix{
\overline{e(\Pp)} \ar[r] \ar[d] & \Gamma_P \backslash \overline{e(\Pp)} \ar[d] \\
\overline{X}_\Pp \ar[r] & \Gamma_{M_\Pp} \backslash \overline{X}_\Pp
}
\end{xy} \]
of bundle maps of manifolds with corners.  The bundle structure of \(\Gamma_P \backslash \overline{e(\Pp)}\) will later be of particular interest.

\begin{theorem}
\label{thm:gammaepfiberbundle}
Let \(\Gamma \subset \Gg(\Q)\) be a neat arithmetic subgroup.  Then the manifold with corners \(\overline{e'(\Pp)} = \Gamma_P \backslash \overline{e(\Pp)}\) has the structure of a real analytic fiber bundle over the manifold with corners \(\Gamma_{M_\Pp} \backslash \overline{X}_\Pp\) with the compact nilmanifold \(\Gamma_{N_P} \backslash N_P\) as typical fiber.
\end{theorem}

Also for later purposes we remark that the Borel--Serre compactification \(\Gamma \backslash \overline{X}\) clearly has a finite CW-structure such that the closed submanifolds \(\overline{e'(\Pp)}\) are subcomplexes.  The bordification \(\overline{X}\) is a regular covering of this finite CW complex with deck transformation group \(\Gamma\), in other words a finite free \(\Gamma\)-CW complex in the sense of \cite{tomDieck:TransformationGroups}*{Section II.1, p.\,98}.  In the sequel we want to assume that \(\overline{X}\) is endowed with this \(\Gamma\)-CW structure as soon as a torsion-free arithmetic subgroup \(\Gamma \subset \Gg(\Q)\) is specified.  Then \fullref{cor:bordificationcontractible} and \fullref{thm:borelserrecompact} say in more abstract terms that the bordification \(\overline{X}\) is a cofinite classifying space \(E \Gamma\).  In fact, something better is true.  The bordification is a model for the classifying space \(\underline{E}\Gamma\) for proper group actions for every general, not necessarily torsion-free, arithmetic subgroup \(\Gamma \subset \Gg(\Q)\).  This means every isotropy group is finite and for every finite subgroup \(\Lambda \subset \Gamma\) the fix point set \(\overline{X}^\Lambda\) is contractible (and in particular nonempty).  This was pointed out in \cite{Adem-Ruan:TwistedOrbifold}*{Remark 5.8, p.\,546} and L.\,Ji thereafter supplied a proof in \cite{Ji:IntegralNovikov}*{Theorem~3.2, p.\,520}.

\section{$\textit{L}^{\text{2}}$-invariants}
\label{sec:l2invariants}

In this section we review \(L^2\)-Betti numbers, Novikov--Shubin invariants and \(L^2\)-torsion of \(\Gamma\)-CW complexes following \cite{Lueck:L2-Invariants}*{Chapters 1--3}.  Let \(\Gamma\) be a discrete countable group.  It acts unitarily from the left on the Hilbert space \(\ell^2 \Gamma\) of square-integrable functions \(\Gamma \rightarrow \C\).  This Hilbert space has a distinguished vector \(e \in \Gamma \subset \ell^2 \Gamma\).  The \(\Gamma\)-equivariant bounded operators \(\mathcal{N}(\Gamma) = \mathcal{B}(\ell^2 \Gamma)^\Gamma\) form a weakly closed, unital \(\ast\)-subalgebra of \(\mathcal{B}(\ell^2 \Gamma)\) called the \emph{group von Neumann algebra} of \(\Gamma\).  This algebra comes endowed with a canonical trace \(\textup{tr}_{\mathcal{N}(\Gamma)}\) given by the matrix coefficient corresponding to the distinguished vector, \(\textup{tr}_{\mathcal{N}(\Gamma)}(f) = \langle f(e), e \rangle\).  The trace \(\textup{tr}_{\mathcal{N}(\Gamma)}\) extends diagonally to positive \(\Gamma\)-equivariant bounded endomorphisms of a direct sum \(\oplus_{k = 1}^n \ell^2 \Gamma\).

Let \(X\) be a finite free \(\Gamma\)-CW-complex in the sense of \cite{tomDieck:TransformationGroups}*{Section II.1, p.\,98}.  Equivalently, \(X\) is a Galois covering of a finite CW-complex with deck transformation group \(\Gamma\).  Let \(C_*(X)\) be the cellular \(\Z \Gamma\)-chain complex.  The \(L^2\)-completion \(C_*^{(2)}(X) = \ell^2 \Gamma \otimes_{\Z \Gamma} C_*(X)\) is called the \emph{\(L^2\)-chain complex}.  The differentials \(c_p \co C_p^{(2)}(X) \rightarrow C_{p-1}^{(2)}(X)\) are \(\Gamma\)-equivariant bounded operators induced from the differentials in \(C_*(X)\).  These define the \(p\)-th \emph{Laplace operator} \(\Delta_p \co C_p^{(2)}(X) \rightarrow C_p^{(2)}(X)\) given by \(\Delta_p = c_{p+1} c_{p+1}^* + c_p^* c_p\).  Let \(\{ E^p_\lambda \}\) be the family of \(\Gamma\)-equivariant spectral projections associated with \(\Delta_p\).  Choosing a cellular basis of \(X\) yields identifications \(C^{(2)}_p(X) = \oplus_{k=0}^{n_p} \ell^2 \Gamma\) where \(n_p\) is the number of equivariant \(p\)-cells in \(X\).  Two such identifications differ by a unitary transformation.  As the trace is constant on unitary conjugacy classes, the following definition is justified.

\begin{definition}
The \(p\)-th \emph{spectral density function} of \(X\) is given by
\[ F_p \co [0, \infty) \to [0,\infty), \quad \lambda \mapsto \textup{tr}_{\mathcal{N}(\Gamma)}(E^p_\lambda). \]
\end{definition}

Spectral density functions are density functions in the measure theoretic sense; they are monotone non-decreasing and right-continuous.

\begin{definition}[Cellular \(L^2\)-invariants]~
\label{def:l2invariants}
\begin{enumerate}[(i)]
\item \label{item:l2invariants:l2betti} The \(p\)-th \emph{\(L^2\)-Betti number} of \(X\) is given by
\[ b^{(2)}_p(X; \mathcal{N}(\Gamma)) = F_p(0) \ \in [0,\infty). \]
\item \label{item:l2invariants:novikovshubin} The \(p\)-th \emph{Novikov-Shubin invariant} of \(X\) is given by
\[ \widetilde{\alpha}_p(X; \mathcal{N}(\Gamma)) = \liminf_{\lambda \rightarrow 0^+} \frac{\log(F_p(\lambda) - F_p(0))}{\log(\lambda)} \ \in [0, \infty] \]
unless \(F_p(\varepsilon) = F_p(0)\) for some \(\varepsilon > 0\) in which case we set \(\widetilde{\alpha}_p(X; \mathcal{N}(\Gamma)) = \infty^+\).
\item \label{item:l2invariants:l2torsion} The \emph{\(L^2\)-torsion} of \(X\) is given by
\[ \rho^{(2)}(X; \mathcal{N}(\Gamma)) = -\frac{1}{2} \sum_{p \geq 0} (-1)^p \,p \int_{0^+}^\infty \log(\lambda)\, \textup{d}F_p(\lambda) \ \in \R \]
where we require \(F_p(0) = 0\) and \(\int_{0^+}^\infty \log(\lambda)\, \textup{d}F_p(\lambda) > -\infty\) for each \(p\).
\end{enumerate}
\end{definition}

Frequently we will suppress \(\mathcal{N}(\Gamma)\) from our notation.  We give some explanations.  The trace of a spectral projection gives the so-called \emph{von Neumann dimension} of its image.  Therefore the \(p\)-th \(L^2\)-Betti number equals the von Neumann dimension of the harmonic \(L^2\)-\(p\)-chains, whence the terminology.  Novikov-Shubin invariants measure how slowly the spectral density function grows in a neighborhood of zero.  The fractional expression is so chosen that it returns \(k\) if \(F_p\) happens to be a polynomial with highest order \(k\).  The value ``\(\infty^+\)'' is just a formal symbol that indicates a spectral gap of \(\Delta_p\) at zero.  We agree that \(\infty^+ > \infty > r\) for all \(r \in \R\).   In the definition of \(L^2\)-torsion we integrate the natural logarithm over the Borel space \((0,\infty)\) with respect to the Lebesgue-Stieltjes measure defined by the density function \(F_p\).  This gives the so-called \emph{Fuglede--Kadison determinant} of \(\Delta_p\).  Note that \(F_p\) equals \(n_p\) after finite time so there  is no issue with divergence to \(+\infty\).  Conjecturally it is also always true that \(\int_{0^+}^\infty \log(\lambda)\, \textup{d}F_p(\lambda) > -\infty\).  This is known if \(\Gamma\) lies in a large class of groups \(\mathcal{G}\) that notably contains all residually finite groups \cite{Schick:L2-Determinant}.  For short we will say that \(X\) is \emph{\(\det\)-\(L^2\)-acyclic} if it satisfies the conditions in (\ref{item:l2invariants:l2torsion}).

For many purposes it is more convenient to work with a finer version of Novikov-Shubin invariants \(\alpha_p(X)\) which we obtain replacing the operator \(\Delta_p\) by \(c_p|_{\textup{im}(c_{p+1})^\perp}\).  We get back the above version by the formula \(\widetilde{\alpha}_p(X) = \frac{1}{2} \min \{\alpha_p(X), \alpha_{p+1}(X) \}\).  Moreover, a finite free \(\Gamma\)-CW-pair \((X,A)\) defines a \emph{relative \(L^2\)-chain complex} \(C_*^{(2)}(X,A)\).  Its Laplacians define the \emph{relative} \(L^2\)-invariants \(b_p^{(2)}(X,A)\), \(\alpha_p(X,A)\) and also \(\rho^{(2)}(X, A)\) provided \((X,A)\) is \(\det\)-\(L^2\)-acyclic.

\begin{theorem}[Selected properties of cellular \(L^2\)-invariants] \label{thm:propofcellularl2}\quad
\begin{enumerate}[(i)]
\item \label{item:propofcellularl2:homotopy} {\bf Homotopy invariance.}  Let \(f \co X \rightarrow Y\) be a weak \(\Gamma\)-homotopy equivalence of finite free \(\Gamma\)-CW-complexes.  Then
\[ b_p^{(2)}(X) = b_p^{(2)}(Y) \quad \text{and} \quad \alpha_p(X) = \alpha_p(Y) \text{ for all } p \geq 0. \]
Suppose that \(X\) or \(Y\) is \(L^2\)-acyclic and that \(\Gamma \in \mathcal{G}\).  Then
\[ \rho^{(2)}(X) = \rho^{(2)}(Y). \]
\item \label{item:propofcellularl2:poincare} {\bf Poincar\'e duality.}  Let the \(\Gamma\)-CW-pair \((X, \partial X)\) be an equivariant triangulation of a free proper cocompact orientable \(\Gamma\)-manifold of dimension \(n\) with possibly empty boundary.  Then
\[ b_p^{(2)}(X) = b_{n-p}^{(2)}(X, \partial X) \quad \text{and} \quad \alpha_p(X) = \alpha_{n+1-p}(X, \partial X). \]
Suppose \(X\) is \(\det\)-\(L^2\)-acyclic.  Then so is \((X, \partial X)\) and
\[ \rho^{(2)}(X) = (-1)^{n+1} \rho^{(2)}(X, \partial X). \]
Thus \(\rho^{(2)}(X) = 0\) if the manifold is even-dimensional and has empty boundary.
\item \label{item:propofcellularl2:fiberbundle} {\bf Euler characteristic and fiber bundles.}  Let \(X\) be a connected finite CW-complex.  Then the classical Euler characteristic \(\chi(X)\) can be computed as
\[ \chi(X) = \sum\limits_{p \geq 0} (-1)^p \,b_p^{(2)}(\widetilde{X}). \]
Let \(F \rightarrow E \rightarrow B\) be a fiber bundle of connected finite CW-complexes.  Assume that the inclusion \(F_b \rightarrow E\) of one (then every) fiber induces an injection of fundamental groups.  Suppose that \(\widetilde{F}_b\) is \(\det\)-\(L^2\)-acyclic.  Then so is \(\widetilde{E}\) and
\[ \rho^{(2)}(\widetilde{E}) = \chi(B) \cdot \rho^{(2)}(\widetilde{F}). \]
\item \label{item:propofcellularl2:aspherical} {\bf Aspherical CW-complexes and elementary amenable groups.}  Let \(X\) be a finite CW-complex with contractible universal covering.  Suppose that \(\Gamma = \pi_1(X)\) is of \(\det \geq 1\)-class and contains an elementary amenable infinite normal subgroup.  Then
\[ b_p^{(2)}(\widetilde{X}) = 0 \text{ for } p \geq 0, \quad \alpha_p(\widetilde{X}) \geq 1 \text{ for } p \geq 1 \quad \text{and} \quad \rho^{(2)}(\widetilde{X}) = 0. \]
\end{enumerate}
\end{theorem}

The proofs are given in \cite{Lueck:L2-Invariants}*{Theorem~1.35, p.\,37, Theorem~2.55 p.\,97, Theorem~3.93, p.\,161, Corollary~3.103, p.\,166, Theorem~3.113, p.\,172, Lemma~13.6, p.\,456}.  The assertion \(\rho^{(2)}(\widetilde{X}) = 0\) in (\ref{item:propofcellularl2:aspherical}) is due to C.\,Wegner \cite{Wegner:L2AsphericalMuenster} who has recently given a slight generalization in \cite{Wegner:L2AsphericalManuscripta}.  We list three more facts that will be of particular importance for our later applications.

\begin{lemma} \label{lemma:novikovshubinmanifoldwithboundary}
Let the \(\Gamma\)-CW-pair \((X, \partial X)\) be an equivariant triangulation of a free proper cocompact orientable \(L^2\)-acyclic \(\Gamma\)-manifold.  Then for each \(p \geq 1\)
\[ \textstyle \frac{1}{2} \min \{\alpha_p(X), \alpha_{n-p}(X) \} \leq \alpha_p(\partial X). \]
\end{lemma}

\begin{proof}
We apply the last inequality of \cite{Lueck:L2-Invariants}*{Theorem 2.20, p.\,84} to the short exact sequence of \(L^2\)-chain complexes of the pair \((X, \partial X)\).  Since \(b^{(2)}_p(X) = 0\), it reduces to
\[ \frac{1}{\alpha_p(\partial X)} \leq \frac{1}{\alpha_p(X)} + \frac{1}{\alpha_{p+1}(X, \partial X)}. \]
The lemma follows because \(\alpha_{p+1}(X, \partial X) = \alpha_{n-p}(X)\) by \fullref{thm:propofcellularl2}\,(\ref{item:propofcellularl2:poincare}).
\end{proof}

Note that the lemma yields \(\widetilde{\alpha}_q(X) \leq \alpha_q(\partial X)\) if \(\dim X = 2q+1\) or \(\dim X = 2q\).  In the latter case it gives in fact more precisely \(\alpha_q(X) \leq 2 \alpha_q(\partial X)\).  The next lemma is stated as Exercise 3.23 in \cite{Lueck:L2-Invariants}*{p.\,209}.

\begin{lemma} \label{lemma:l2torsionmanifoldwithboundary}
Let the \(\Gamma\)-CW-pair \((X, \partial X)\) be an equivariant triangulation of a free proper cocompact orientable \(\Gamma\)-manifold of even dimension.  Assume \(X\) is \(\det\)-\(L^2\)-acyclic.  Then so is \(\partial X\) and
\[ \textstyle \rho^{(2)}(X) = \frac{1}{2} \rho^{(2)}(\partial X). \]
\end{lemma}

Finally, we recall that \(L^2\)-torsion has the same additivity property as the Euler characteristic \cite{Lueck:L2-Invariants}*{Theorem~3.93(2), p.\,161}.

\begin{lemma} \label{lemma:l2torsionpushouts}
Consider the pushout of finite free \(\Gamma\)-CW complexes
\[
\begin{xy}
\xymatrix{
X_0 \ar[r]^{j_2} \ar[d]_{j_1} & X_2 \ar[d]\\
X_1 \ar[r] & X
}
\end{xy}
\]
where \(j_1\) is an inclusion of a \(\Gamma\)-subcomplex, \(j_2\) is cellular and \(X\) carries the induced \(\Gamma\)-CW-structure.  Assume that \(X_i\) is \(\det\)-\(L^2\)-acyclic for \(i = 0, 1, 2\).  Then so is \(X\) and
\[ \rho^{(2)}(X) = \rho^{(2)}(X_1) + \rho^{(2)}(X_2) - \rho^{(2)}(X_0). \]
\end{lemma}

\(L^2\)-invariants, being homotopy invariants by \fullref{thm:propofcellularl2}\,(\ref{item:propofcellularl2:homotopy}), yield invariants for groups whose classifying spaces have a finite CW-model \(B\Gamma\).   For this purpose we set \(b_p^{(2)}(\Gamma) = b_p^{(2)}(E\Gamma; \mathcal{N}(\Gamma))\) as well as \(\alpha_p(\Gamma) = \alpha_p(E\Gamma; \mathcal{N}(\Gamma))\).  We say that \(\Gamma\) is \emph{\(\det\)-\(L^2\)-acyclic} if \(E\Gamma\) is, and set \(\rho^{(2)}(\Gamma) = \rho^{(2)}(E\Gamma; \mathcal{N}(\Gamma))\) in that case.  In fact, \(L^2\)-Betti numbers have been generalized to arbitrary \(\Gamma\)-spaces and thus to arbitrary groups \citelist{\cite{Cheeger-Gromov:L2-cohomology} \cite{Lueck:ArbitraryModules}}.  Novikov--Shubin invariants can likewise be defined for general groups \cite{Lueck-Reich-Schick:ArbitraryNovikovShubin}.  So we shall allow ourselves to talk about \(b^{(2)}_p(\Gamma)\), \(\alpha_p(\Gamma)\) and \(\widetilde{\alpha}_p(\Gamma)\) for any countable discrete group \(\Gamma\).  Only for the \(L^2\)-torsion such a generalization has not (yet) been given.

If \(M\) is a cocompact free proper Riemannian \(\Gamma\)-manifold without boundary, there is a parallel theory of \emph{analytic \(L^2\)-invariants} of \(M\), exploiting the analytic Laplacian \(\Delta_p^a\) acting on square integrable \(p\)-forms on \(M\) \cite{Lueck:L2-Invariants}*{Sections 1.3, 2.3, 3.5}.  Since \(\Delta_p^a\) is an unbounded operator, some more technical effort is necessary in particular to handle analytic \(L^2\)-torsion.  If \(M\) comes equipped with a finite equivariant triangulation, then cellular and analytic \(L^2\)-invariants agree.  The result is due to J.\,Dodziuk for the \(L^2\)-Betti numbers \cite{Dodziuk:DeRhamHodge}, to A.\,V.\,Efremov for the Novikov--Shubin invariants \cite{Efremov:Novikov} and lastly to D.\,Burghelea, L.\,Friedlander, T.\,Kappeler and P.\,McDonald for the \(L^2\)-torsion \cite{Burghelea-et-al:Torsion}.  This bridge between topological and analytic methods makes \(L^2\)-invariants powerful.  On the one hand from the analytic definition it is not at all obvious that \(L^2\)-invariants are homotopy invariants.  On the other hand the analytic approach can give access to computations if the Riemannian structure is particularly nice.  This definitely applies to the case of a \emph{symmetric space of noncompact type}, \(M = G/K\) for a connected semisimple Lie group \(G\) with maximal compact subgroup \(K\).  Then \(M\) is a finite \(E\Gamma\) for every torsion-free uniform lattice \(\Gamma \subset G\).  Recall that the \emph{deficiency} of \(G\) is given by \(\delta(G) = \rank_\C(G) - \rank_\C(K)\).

\begin{theorem}[\(L^2\)-invariants of uniform lattices] \label{thm:l2symmetricspaces}
Let \(\Gamma \subset G\) be a uniform lattice and set \(m = \delta(G)\) and \(n = \dim (M)\).
\begin{enumerate}[(i)]
\item \label{item:l2symmetricspaces:l2betti} We have \(b^{(2)}_p(\Gamma) \neq 0\) if and only if \(m = 0\) and \(n = 2p\).
\item \label{item:l2symmetricspaces:novikov} We have \(\alpha_p(\Gamma) = \infty^+\) unless \(m > 0\) and \(p \in [\frac{n-m}{2} + 1, \frac{n+m}{2}]\) in which case \(\textstyle \alpha_p(\Gamma) = m\).
\item \label{item:l2symmetricspaces:l2torsion} Assume that \(\Gamma\) is torsion-free.  We have \(\rho^{(2)}(\Gamma) \neq 0\) if and only if \(m = 1\).
\end{enumerate}
\end{theorem}

Part (\ref{item:l2symmetricspaces:l2betti}) can already be found in \cite{Borel:L2-cohomology}.  Parts (\ref{item:l2symmetricspaces:novikov}) and (\ref{item:l2symmetricspaces:l2torsion}) are due to M.\,Olbrich \cite{Olbrich:L2-InvariantsLocSym} generalizing previous work of J.\,Lott \cite{Lott:HeatKernels} and E.\,Hess--T.\,Schick \cite{Hess-Schick:L2Hyperbolic}.  Methods involve \((\fr{g}, K)\)-cohomology as well as the Harish-Chandra--Plancherel Theorem.  Formulas for the nonzero values of \(L^2\)-Betti numbers and \(L^2\)-torsion involving the geometry of the compact dual of \(M\) are also given in \cite{Olbrich:L2-InvariantsLocSym}.  We note that \(n-m\) (thus \(n+m\)) is always even and positive.

For \(L^2\)-Betti numbers nothing new happens in the case of a nonuniform lattice \(\Gamma \subset G\).

\begin{theorem} \label{thm:l2bettiofgenerallattices}
Let \(\Gamma \subset G\) be any lattice and set \(m = \delta(G)\) and \(n = \dim (M)\).  We have \(b^{(2)}_p(\Gamma) \neq 0\) if and only if \(m = 0\) and \(n = 2p\).
\end{theorem}

This is already contained in the work of Cheeger--Gromov \cite{Cheeger-Gromov:BoundsVonNeumann} who consider compact exhaustions of certain finite-volume manifolds.  A more conceptual line of reasoning uses that \(G\) possesses uniform lattices which are all measure equivalent to \(\Gamma\).  Hence the result follows from a proportionality theorem of D.\,Gaboriau \cite{Gaboriau:Invariantsl2}*{Th\'eor\`eme 6.3, p.\,95}.

\section{$\textit{L}^{\text{2}}$-invariants of the Borel-Serre bordification}
\label{sec:l2ofborelserre}

Let us recall that G.\,Margulis showed that taking integer points of algebraic \(\Q\)-groups is essentially the only way to produce lattices in higher rank Lie groups.  A lattice \(\Gamma\) in a connected semisimple Lie group \(G\) without compact factors is called \emph{reducible} if \(G\) admits infinite connected normal subgroups \(H\) and \(H'\) such that \(G = HH'\), such that \(H \cap H'\) is discrete and such that \(\Gamma / (\Gamma \cap H)(\Gamma \cap H')\) is finite. Otherwise \(\Gamma\) is called \emph{irreducible}.  Two groups are called \emph{abstractly commensurable} if they have isomorphic subgroups of finite index.

\begin{theorem}[Margulis arithmeticity] \label{thm:margulisarithmeticity}
Let \(G\) be a connected semisimple linear Lie group of \(\rank_\R(G) > 1\) without compact factors.  Let \(\Gamma \subset G\) be an irreducible lattice.  Then there is a connected semisimple linear \(\Q\)-group \(\Hh\) such that \(\Gamma\) and \(\Hh(\Z)\) are abstractly commensurable and such that \(G\) and \(\Hh(\R)\) define isometric symmetric spaces.
\end{theorem}

The standard formulation of Margulis arithmeticity is slightly different \cite{Margulis:Arithmeticity}*{Theorem~1, p.\,97}; see \cite{Kammeyer:L2-invariants}*{Corollary 4.4, p.\,33} for the conclusion of our version.  W.\,L\"uck, H.\,Reich and T.\,Schick have shown in \cite{Lueck-Reich-Schick:ArbitraryNovikovShubin}*{Theorem~3.7.1} that abstractly commensurable groups have equal Novikov--Shubin invariants.  Therefore all irreducible lattices in higher rank semisimple Lie groups are covered when we work for the moment with arithmetic subgroups of connected semisimple linear algebraic \(\Q\)-groups.  Before we come to the proof of \fullref{thm:novikovshubinqrankone}, we need to recall the following definition for a compactly generated locally compact group \(H\) with compact generating set \(V \subset H\) and Haar measure \(\mu\) (compare \cite{Guivarch:Croissance}).

\begin{definition}
\label{def:polynomialgrowth}
The group \(H\) has \emph{polynomial growth of order \(d(H) \geq 0\)} if
\[ d(H) = \inf \left\{ k > 0 \co \limsup_{n \rightarrow \infty} \textstyle \frac{\mu(V^n)}{n^k} < \infty \right\}. \]
\end{definition}

This definition is independent of the choice of \(V\) and of rescaling \(\mu\) \cite{Guivarch:Croissance}*{p.\,336}.  If \(H\) is discrete and \(V\) is a finite symmetric generating set, we get back the familiar definition in terms of metric balls in the Cayley graph defined by word lengths.  Let us recall the result we want to prove.

\medskip
{\bf \fullref{thm:novikovshubinqrankone}}\qua
{\sl Let \(\Gg\) be a connected semisimple linear algebraic \(\Q\)-group.  Suppose that \(\rank_\Q(\Gg) = 1\) and \(\delta(\Gg(\R)) > 0\).  Let \(\Pp \subset \Gg\) be a proper rational parabolic subgroup.  Then for every arithmetic lattice \(\Gamma \subset \Gg(\Q)\)}
\[ \widetilde{\alpha}_q(\Gamma) \leq \delta(M_\Pp) + d(N_P). \]

Here \(q\) is the \emph{middle dimension} of \(X = \Gg(\R) / K\), so either \(\dim X = 2q + 1\) or \(\dim X = 2q\).  The deficiency of a reductive Lie group \(G'\) is defined as \(\delta(G') = \rank_\C (G') - \rank_\C (K')\) for a maximal compact subgroup \(K' \subset G'\) as in the case of semisimple groups.  The deficiency of \(G'\) is also known as the \emph{fundamental rank} \(\textup{f-rank}(X')\) of the associated symmetric space \(X' = G'/K'\).  Note that \(\Gg\) trivially satisfies conditions (\ref{cond:anisotropiccenter}) and (\ref{cond:centralizermeetscomponents}) of \fullref{sec:arithmeticsubgroups}.  Since \(\rank_\Q(\Gg) = 1\), all proper rational parabolic subgroups are conjugate under \(\Gg(\Q)\) so that the constant \(\delta(M_\Pp) + d(N_P)\) only depends on \(\Gg\).  One example of a group \(\Gg\) as in \fullref{thm:novikovshubinqrankone} is of course \(\Gg = \textup{SO}(2n+1,1; \C)\).  But the point of \fullref{thm:novikovshubinqrankone} is that no restriction is made on the real rank of \(\Gg\) and we will consider groups \(\Gg\) with higher real rank in \fullref{ex:quadraticform} after proving the theorem.  The proof will require an estimation of Novikov--Shubin invariants of the boundary components \(e(\Pp) = N_P \times X_\Pp\) of the Borel--Serre bordification \(\overline{X}\).  Since a product formula is available for Novikov--Shubin invariants, the calculation eventually reduces to \fullref{thm:l2symmetricspaces}\,(\ref{item:l2symmetricspaces:novikov}) and the following theorem due to M.\,Rumin \cite{Rumin:AroundHeat}*{Theorem~3.13, p.\,144}, see also \cite{Rumin:DiffGeoOnCC}*{Theorem~4, p.\,990}.  

\begin{theorem}[M.\,Rumin] \label{thm:rumin}
Let \(N\) be a simply connected nilpotent Lie group whose Lie algebra \(\fr{n}\) comes with a grading \(\fr{n} = \bigoplus_{k=1}^r \fr{n}_k\).  Assume that \(N\) possesses a uniform lattice \(\Gamma_N\).  Then for each \(p = 1, \ldots, \dim N\)
\[ 0 < \alpha_p(N; \mathcal{N}(\Gamma_N)) \leq \sum_{k=1}^r k \dim \fr{n}_k. \]
\end{theorem}

In fact, Rumin gives a finer pinching than the above, which in special cases gives precise values.  For example \(\alpha_2(N; \mathcal{N}(\Gamma_N)) = \sum_{k=1}^r k \dim \fr{n}_k\) if \(N\) is \emph{quadratically presented} \cite{Rumin:AroundHeat}*{Section 4.1, p.\,146}.

\begin{corollary} \label{cor:novikovshubinofnilpotent}
Let \(\Pp \subset \Gg\) be a proper rational parabolic subgroup.  Then for every torsion-free arithmetic subgroup \(\Gamma \subset \Gg(\Q)\) and each \(p = 1, \ldots, \dim N_P\) we have
\[ \alpha_p(N_P; \mathcal{N}(\Gamma_{N_P})) \leq d(N_P). \]
\end{corollary}

\begin{proof}
At the end of \fullref{sec:rationalparabolicsubgroups} we have seen that the Lie algebra \(\fr{n}_P\) of \(N_P\) is conjugate to a standard \(\fr{n}_I = \bigoplus_{\alpha \in \Sigma} \fr{n}_{P, \alpha}\) and thus graded by the lengths of parabolic roots.  Since \([\fr{n}_{P, \alpha}, \fr{n}_{P, \beta}] \subset \fr{n}_{P, \alpha + \beta}\) by Jacobi identity, this graded algebra can be identified with the graded algebra associated with the filtration of \(\fr{n}_P\) coming from its lower central series.  It thus follows from \cite{Guivarch:Croissance}*{Th\'{e}or\`{e}me~II.1, p.\,342} that the weighted sum appearing in \fullref{thm:rumin} equals the degree of polynomial growth of \(N_P\).
\end{proof}

\begin{proposition} \label{prop:novikovshubinofboundarycomponent}
Suppose \(\rank_\Q(\Gg) = 1\).  Then for every proper rational parabolic subgroup \(\Pp \subset \Gg\) and every torsion-free arithmetic subgroup \(\Gamma \subset \Gg(\Q)\) we have
\[ \alpha_q(e(\Pp); \mathcal{N}(\Gamma_P)) \leq \textup{f-rank}(X_\Pp) + d(N_P). \]
\end{proposition}

\begin{proof} Fix such \(\Pp \subset \Gg\) and \(\Gamma \subset \Gg(\Q)\).  We mentioned below \fullref{def:neat} that \(\Gamma\) possesses a neat and thus torsion-free subgroup of finite index.  It induces a neat subgroup of finite index of \(\Gamma_P\).  Since Novikov--Shubin invariants remain unchanged for finite index subgroups, we may assume that \(\Gamma\) itself is neat.  Thus \(\Gamma_{M_\Pp}\) acts freely on \(X_\Pp\).  As \(\rank_\Q(\Gg) = 1\), every proper rational parabolic subgroup is minimal (and maximal).  So the boundary component \(e(\Pp)\) is closed as we observed below \fullref{prop:closureofep}. Therefore the \(\Gamma_P\)-action on \(e(\Pp)\) is cocompact.  Since also \(\Gamma_{N_P} \times \Gamma_{M_\Pp}\) acts cocompactly, \cite{Lueck:L2-Invariants}*{Theorem 3.183, p.\,201} implies
\[ \alpha_q(e(\Pp); \mathcal{N}(\Gamma_P)) = \alpha_q(N_P \times X_\Pp; \mathcal{N}(\Gamma_{N_P} \times \Gamma_{M_\Pp})). \]
This observation enables us to apply the \emph{product formula} for Novikov--Shubin invariants \cite{Lueck:L2-Invariants}*{Theorem~2.55(3), p.\,97}.  It says that \(\alpha_q(N_P \times X_\Pp; \mathcal{N}(\Gamma_{N_P} \times \Gamma_{M_\Pp}))\) equals the minimum of the union of the four sets
\begin{align*}
\{ &\alpha_{i+1}(N_P) + \alpha_{q-i}(X_\Pp) \co \ i = 0, \ldots, q-1 \},\\
\{ &\alpha_i(N_P) + \alpha_{q-i}(X_\Pp) \co \ i = 1, \ldots, q-1 \},\\
\{ &\alpha_{q-i}(X_\Pp) \co \ i = 0, \ldots, q-1, \ b^{(2)}_i(N_P) > 0 \},\\
\{ &\alpha_i(N_P) \co \ i = 1, \ldots, q, \ b^{(2)}_{q-i}(X_\Pp) > 0 \}.
\end{align*}
We need to discuss one subtlety here.  Applying the product formula requires us to verify that both \(N_P\) and \(X_\Pp\) have the \emph{limit property}.  This means that ``\(\liminf\)'' in \fullref{def:l2invariants}\,(\ref{item:l2invariants:novikovshubin}) equals ``\(\limsup\)'' of the same expression.  But this follows from the explicit calculations in \cite{Rumin:DiffGeoOnCC} and \cite{Olbrich:L2-InvariantsLocSym}.  Note that the third set above is actually empty because of \fullref{thm:propofcellularl2}\,(\ref{item:propofcellularl2:aspherical}).  The group \(\Mm_\Pp = \mathbf{Z}_\Pp \Mm_\Pp'\) is the almost direct product of its center \(\mathbf{Z}_\Pp\) and the derived subgroup \(\Mm_\Pp' = [\Mm_\Pp, \Mm_\Pp]\) which is semisimple.  Accordingly, the boundary symmetric space \(X_\Pp = X_\Pp^{\text{Eucl}} \times X_\Pp^{\text{nc}}\) is the product of a Euclidean symmetric space and a symmetric space of noncompact type.  Clearly \(\textup{f-rank}(X_\Pp^{\textup{Eucl}}) = \dim X_\Pp^{\textup{Eucl}}\) so that
\[ \textup{f-rank}(X_\Pp) = \textup{f-rank}(X_\Pp^{\textup{Eucl}} \times X_\Pp^{\textup{nc}}) = \dim X_\Pp^{\text{Eucl}} + \textup{f-rank}(X_\Pp^{\textup{nc}}). \]
As \(\textup{s-rank}(\Pp) = 1\) we get \(\dim e(\Pp) = \dim X - 1\) with \(\dim X = 2q\) or \(\dim X = 2q + 1\).  Let us set \(n = \dim N_P\), hence \(\dim X_\Pp = \dim X - 1 - n\).  Now we distinguish two cases.  First we assume that \(\textup{f-rank}(X_\Pp) = 0\).  Then \(X_\Pp = X_\Pp^{\text{nc}}\) is even-dimensional and we obtain from \fullref{thm:l2symmetricspaces}\,(\ref{item:l2symmetricspaces:l2betti}) that \(b^{(2)}_{q - \lceil \frac{n}{2} \rceil}(X_\Pp) > 0\).  Here for a real number \(a \in \R\) we denote by \(\lceil a \rceil\) and \(\lfloor a \rfloor\) the smallest integer not less than \(a\) and the largest integer not more than \(a\), respectively.  Therefore the Novikov--Shubin invariant \(\alpha_{\lceil \frac{n}{2} \rceil} (N_P)\) appears in the fourth set above and is bounded by \(d(N_P)\) according to \fullref{cor:novikovshubinofnilpotent}.  Now let us assume \(\textup{f-rank}(X_\Pp) > 0\).  We compute \(q - \lceil \frac{n}{2} \rceil = \lfloor \frac{\dim X_\Pp + 1}{2} \rfloor\) if \(\dim X = 2q\) and \(q - \lfloor \frac{n}{2} \rfloor = \lceil \frac{\dim X_\Pp}{2} \rceil\) if \(\dim X = 2q+1\).  We claim that both values lie in the interval \([ \frac{1}{2}(\dim X_\Pp - \textup{f-rank}(X_\Pp)) + 1, \frac{1}{2}(\dim X_\Pp + \textup{f-rank}(X_\Pp)) ]\).  This is clear if \(\dim X_\Pp\) is odd because then both values equal \(\frac{\dim X_\Pp + 1}{2}\) which is the arithmetic mean of the interval limits.  If on the other hand \(\dim X_\Pp\) is even, then both values equal \(\frac{\dim X_\Pp}{2}\).  The fundamental rank \(\textup{f-rank}(X_\Pp)\) is then likewise even and thus \(\textup{f-rank}(X_\Pp) \geq 2\).  Therefore \(\frac{1}{2}(\dim X_\Pp - \textup{f-rank}(X_\Pp)) + 1 \leq \frac{\dim X_\Pp}{2}\) and the claim is verified.  It follows from \cite{Lueck:L2-Invariants}*{equation (5.14), p.\,230} that in the two cases \(\alpha_{q - \lceil \frac{n}{2} \rceil} (X_\Pp)\) and \(\alpha_{q - \lfloor \frac{n}{2} \rfloor}(X_\Pp)\) are bounded by \(\text{f-rank}(X_\Pp^{\text{nc}}) + \dim X_\Pp^{\text{Eucl}} = \textup{f-rank}(X_\Pp)\).  Moreover \(\alpha_{\lceil \frac{n}{2} \rceil} (N_P) \leq d(N_P)\) and \(\alpha_{\lfloor \frac{n}{2} \rfloor} (N_P) \leq d(N_P)\) again by \fullref{cor:novikovshubinofnilpotent} so that either the number \(\alpha_{\lceil \frac{n}{2} \rceil} (N_P) + \alpha_{q - \lceil \frac{n}{2} \rceil} (X_\Pp)\) or the number \(\alpha_{\lfloor \frac{n}{2} \rfloor} (N_P) + \alpha_{q - \lfloor \frac{n}{2} \rfloor} (X_\Pp)\) appears in the second of the four sets above and both are bounded by \(d(N_P) + \textup{f-rank}(X_\Pp)\).  So in any case we conclude \(\alpha_q(e(\Pp)) \leq \textup{f-rank}(X_\Pp) + d(N_P)\).
\end{proof}

We make one last elementary observation to prepare the proof of \fullref{thm:novikovshubinqrankone}.

\begin{lemma} \label{lemma:injectivefundamentalgroup}
Let the discrete group \(\Gamma\) act freely and properly on the path-connected space \(X\).  Let \(Y \subset X\) be a simply connected subspace which is invariant under the action of a subgroup \(\Lambda \leq \Gamma\).  Then the induced homomorphism \(\Lambda = \pi_1(\Lambda \backslash Y) \rightarrow \pi_1(\Gamma \backslash X)\) is injective.
\end{lemma}

\begin{proof}
From covering theory we obtain a commutative diagram of groups
\[
\begin{xy}
\xymatrix{
\pi_1(\Lambda \backslash Y) \ar[r] \ar[d] & \Lambda \ar[d] \\
\pi_1(\Gamma \backslash X) \ar[r] & \Gamma.
}
\end{xy}
\]
The upper map is an isomorphism and the right hand map is injective.  So the left hand map must be injective as well.
\end{proof}

\begin{proof}[Proof of \fullref{thm:novikovshubinqrankone}]
Again by Selberg's Lemma and stability of Novikov-Shubin invariants for finite index subgroups \cite{Lueck-Reich-Schick:ArbitraryNovikovShubin}*{Theorem~3.7.1}, we may assume that \(\Gamma\) is torsion-free.  The bordification \(\overline{X}\) is \(L^2\)-acyclic by \fullref{thm:l2bettiofgenerallattices}.  According to \fullref{lemma:novikovshubinmanifoldwithboundary} we thus have \(\widetilde{\alpha}_q(\overline{X}) \leq \alpha_q(\partial \overline{X})\).  Recall from \eqref{eqn:borelserreboundary} that the Borel--Serre boundary \(\partial \overline{X} = \bigcupdot_{\Pp \subsetneq \Gg} e(\Pp)\) is given by the disjoint union of all boundary components of proper rational parabolic subgroups.  Since \(\rank_\Q(\Gg) = 1\), every proper rational parabolic subgroup is minimal so all the boundary components are closed.  As \(\overline{X}\) is normal (\(T_4\)), the Borel--Serre boundary is in fact the coproduct \(\partial \overline{X} = \coprod_{\Pp \text{ min}} e(\Pp)\) of all boundary components of minimal rational parabolic subgroups.  \fullref{prop:decompofborelserre} implies that there is a finite system of representatives \(\Pp_1, \ldots, \Pp_k\) of \(\Gamma\)-conjugacy classes of minimal rational parabolic subgroups which give the decomposition \(\Gamma \backslash \partial \overline{X} = \coprod_{i=1}^k e'(\Pp_i)\).  It thus follows from \fullref{lemma:injectivefundamentalgroup} applied to each \(e(\Pp_i) \subset \overline{X}\) and \(\Gamma_{P_i} \leq \Gamma\) that \(\partial \overline{X} = \coprod_{i=1}^k e(\Pp_i) \times_{\Gamma_{P_i}} \Gamma\).  According to \cite{Lueck:L2-Invariants}*{Lemma~2.17(3), p.\,82} we obtain \(\alpha_q(\partial \overline{X}) = \min_i\, \{ \alpha_q(e(\Pp_i) \times_{\Gamma_{P_i}} \Gamma)\}\).  Since the minimal rational parabolic subgroups \(\Pp_1, \ldots, \Pp_k\) are \(\Gg(\Q)\)-conjugate, we have in fact \(\alpha_q(\partial \overline{X}) = \alpha_q(e(\Pp_1) \times_{\Gamma_{P_1}} \Gamma)\).  The induction principle for Novikov--Shubin invariants \cite{Lueck:L2-Invariants}*{Theorem~2.55(7), p.\,98} in turn says that \(\alpha_q(e(\Pp_1) \times_{\Gamma_{P_1}} \!\Gamma; \,\mathcal{N}(\Gamma)) = \alpha_q(e(\Pp_1); \,\mathcal{N}(\Gamma_{P_1}))\) which is bounded from above by \(\textup{f-rank}(X_{\Pp_1}) + d(N_{P_1})\) according to \fullref{prop:novikovshubinofboundarycomponent}.
\end{proof}

For the following example we assume some familiarity with the classification theory of semisimple algebraic groups over non-algebraically closed fields as outlined in \cite{Tits:Classification}.

\begin{example}
\label{ex:quadraticform}
Upon discussions with F.\,Veneziano and M.\,Wiethaup we have come up with the family of senary diagonal quadratic forms
\[ Q^p = \langle 1, 1, 1, -1, -p, -p \rangle \]
over \(\Q\) where \(p\) is a prime congruent to  \(3 \!\mod 4\).  Let \(\Gg^p = \textup{SO}(Q^p; \C)\) be the \(\Q\)-subgroup of \(\textup{SL}(6; \C)\) of matrices preserving \(Q^p\).  By Sylvester's law of inertia, the groups \(\Gg^p\) are \(\R\)-isomorphic to \(\textup{SO}(3,3; \C)\), so that \(\Gg^p(\R) \cong \textup{SO}(3,3)\) which has deficiency one.  Over \(\Q\) there is an obvious way of splitting off one hyperbolic plane,
\[ Q^p = \langle 1, -1 \rangle \perp \langle 1, 1, -p, -p \rangle, \]
but the orthogonal complement \(\langle 1, 1, -p, -p \rangle\) is \(\Q\)-anisotropic.  To see this, recall from elementary number theory that if a prime congruent to \(3 \!\mod 4\) divides a sum of squares, then it must divide each of the squares.  It thus follows from infinite descent that the Diophantine equation \(x_1^2 + x_2^2 = p(x_3^2 + x_4^2)\) has no integer and thus no rational solution other than zero.  Therefore \(\rank_\Q(\Gg^p) = 1\) and \(\Gg^p\) satisfies the conditions of \fullref{thm:novikovshubinqrankone}.  The group \(\Gg^p\) is \(\overline{\Q}\)-isomorphic to \(\textup{SO}(6; \C)\) which accidentally has \(\textup{SL}(4;\C)\) as a double cover and thus is of type \(A_3\).  Note that the hyperbolic plane in the above decomposition of \(Q^p\) gives an obvious embedding of a one-dimensional \(\Q\)-split torus \(\Ss\) into \(\Gg^p\).   Let \(\mathbf{T} \subset \Gg^p\) be a maximal torus containing \(\Ss\).  Then from the tables in \cite{Tits:Classification} , one sees that \(\Gg^p\) can only have one of the following two \emph{Tits indices}.
\centerline{\includegraphics[width=5.1cm]{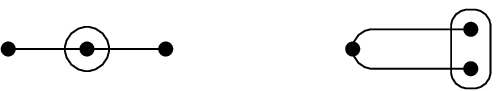}}
The Tits index is given by the Dynkin diagram of the root system \(\Phi(\Gg^p, \mathbf{T})\) where simple roots in the same \(\textup{Gal}(\overline{\Q}/\Q)\)-orbit are drawn close to one another and where the \emph{distinguished orbits}, consisting of roots that do not restrict to zero on \(\Ss\), are circled.  To find out which of the above indices is correct, let \(\Pp \subset \Gg_p\) be a minimal parabolic subgroup corresponding to a choice of positive \emph{restricted roots} of \(\Gg^p\) with respect to \(\Ss = \Ss_\Pp\).  The centralizer \(\mathcal{Z}_{\Gg^p}(\Ss_\Pp) =  \Ss_\Pp \Mm_\Pp = \Ss_\Pp \mathbf{Z}_\Pp \Mm_\Pp'\) obviously contains a \(\Q\)-subgroup that is \(\R\)-isomorphic to \(\textup{SO}(2,2;\C)\) so that \(\textup{SO}(4;\C) \subset \Mm_\Pp'\)as a \(\C\)-embedding.  Because of the exceptional isomorphism \(D_2 = A_1 \times A_1\), the Dynkin diagram of \(\Mm_\Pp'\) must contain two disjoint nodes.  But we obtain the Dynkin diagram and in fact the Tits index of \(\Mm_\Pp'\) by removing the distinguished orbits.  Therefore we see that only the left hand Tits index can correspond to \(\Gg^p\).  Since it is of \emph{inner type} \cite{Tits:Classification}, the center \(\mathbf{Z}_\Pp\) of \(\Mm_\Pp\) is trivial and in fact \(\Mm_\Pp = \Mm_\Pp' \cong_\R \textup{SO}(2,2;\C)\).  Thus \(\delta(M_\Pp) = \delta(\textup{SO}(2,2)) = \delta(\textup{SL}(2;\R) \times \textup{SL}(2,\R)) = 0\).

Now we explain how to compute the number \(d(N_P)\).  The Lie algebra \(\fr{n}_P\) of \(N_P\) has the decomposition \(\fr{n}_P = \bigoplus_{\alpha \in \Sigma} \fr{n}_{P, \alpha}\) as we saw at the end of \fullref{sec:rationalparabolicsubgroups} so that \(\fr{n}_P\) is graded by parabolic root lengths.  In view of the formula in \fullref{thm:rumin} it only remains to determine \(\Sigma\) and the \emph{multiplicities} \(m_\alpha\) given by the dimensions of the root spaces \(\fr{n}_{P, \alpha}\).  The root system \(\Phi(\Gg, \mathbf{T})\) is three-dimensional so that we can see this data visually in \fullref{fig:roota3res}.
\begin{figure}[htb]
\centering
\includegraphics[width=5.5cm]{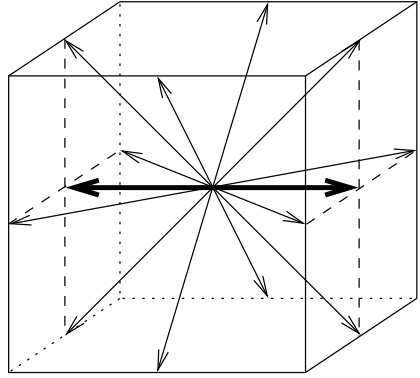}
\caption{The root system of type \(A_3\) with the restricted root system depicted by thick arrows.}
\label{fig:roota3res}
\end{figure}
In the Tits index of \(\Gg^p\), the left hand node corresponds to the arrow pointing up front, the center node corresponds to the arrow pointing down right and the right hand node corresponds to the arrow pointing up rear.  Since both the left and right nodes of the Tits index do not lie in distinguished orbits, the subspace \(X_\Q(\Ss_\Pp) \otimes_\Z \R\) is given by the intersection of the planes orthogonal to their corresponding arrows which is the line going through the centers of the left face and right face of the cube.  It follows that the restricted root system \(\Phi(\Gg^p, \Ss_\Pp)\) is of type \(A_1\) and that four roots of \(\Phi(\Gg^p, \mathbf{T})\) restrict to each of the two roots in \(\Phi(\Gg^p, \Ss_\Pp)\).  Thus we have only one root of length one and multiplicity four in \(\Sigma = \Phi^+(\Gg_\Pp, \Ss_\Pp)\) which gives \(d(N_P) = 4\).  The symmetric space of \(\Gg^p(\R)\) has dimension nine, so \fullref{thm:novikovshubinqrankone} gives
\[ \widetilde{\alpha}_4(\Gg^p(\Z)) \leq 4. \]
Note that the bound is uniform in \(p\) even though the quadratic forms \(Q^p\) and hence the groups \(\Gg^p\) are definitely not mutually \(\Q\)-isomorphic.  Since \(\textup{SO}(6; \C)\) is doubly covered by \(\textup{SL}(4; \C)\), we can take the preimage of \(\Gg^p(\Z)\) to get nonuniform lattices in \(\textup{SL}(4; \R)\) whose fourth Novikov-Shubin invariant is equally bounded by four.
\end{example}

Now we turn our attention to \(L^2\)-torsion.  Recall that \(L^2\)-torsion is only defined for groups which are \(\det\)-\(L^2\)-acyclic.  For a lattice \(\Gamma \subset G\) in a semisimple Lie group we have \(\Gamma \in \mathcal{G}\) so that this is equivalent to \(\delta(G) > 0\) by \fullref{thm:l2bettiofgenerallattices}.  Among the rank one simple Lie groups, the only groups with positive deficiency are \(G = \textup{SO}^0(2n+1,1)\) which have been treated by W.\,L\"uck and T.\,Schick in \cite{Lueck-Schick:Hyperbolic}.  For higher rank Lie groups, we again have Margulis arithmeticity available so that the following Theorem will be enough to cover general lattices in even deficiency groups as we will see subsequently.

\begin{theorem} \label{thm:l2torsionlattices}
Let \(\Gg\) be a connected semisimple linear algebraic \(\Q\)-group.  Suppose that \(\Gg(\R)\) has positive, even deficiency.  Then every torsion-free arithmetic lattice \(\Gamma \subset \Gg(\Q)\) is \(\det\)-\(L^2\)-acyclic and
\[ \rho^{(2)}(\Gamma) = 0. \]
\end{theorem}

Note that in the odd deficiency case, Borel and Serre have proved correspondingly that \(\chi(\Gamma) = 0\) in \cite{Borel-Serre:Corners}*{Proposition 11.3, p.\,482}.  The core idea will also prove successful for the proof of \fullref{thm:l2torsionlattices} though various technical difficulties arise owed to the considerably more complicated definition of \(L^2\)-torsion.  A combinatorial argument will reduce the calculation of the \(L^2\)-torsion of \(\overline{X} = \bigcup_{\Pp \subseteq \Gg} \overline{e(\Pp)}\) to the calculation of the \(L^2\)-torsion of the manifolds with corners \(\overline{e(\Pp)}\) for proper rational parabolic subgroups \(\Pp \subset \Gg\) which form the boundary \(\partial \overline{X}\) of the bordification.  This in turn is settled by the following proposition.

\begin{proposition} \label{prop:l2torsionofboundarycomponent}
Let \(\Pp \subset \Gg\) be a proper rational parabolic subgroup.  Then for every torsion-free arithmetic subgroup \(\Gamma \subset \Gg(\Q)\) the finite free \(\Gamma_P\)-CW complex \(\overline{e(\Pp)} \subset \overline{X}\) is \(\det\)-\(L^2\)-acyclic and \(\rho^{(2)}(\overline{e(\Pp)}; \mathcal{N}(\Gamma_P)) = 0\).
\end{proposition}

\begin{proof}
\(L^2\)-torsion is multiplicative under finite coverings \cite{Lueck:L2-Invariants}*{Theorem~3.96(5), p.\,164} so that similar to the proof of \fullref{prop:novikovshubinofboundarycomponent}, we may assume that \(\Gamma\) is neat.  We have already remarked below \fullref{thm:borelserrecompact} that \(e(\Pp)\), hence its closure \(\overline{e(\Pp)}\), is a \(\Gamma_P\)-invariant subspace of the bordification \(\overline{X}\).  So \(\overline{e(\Pp)}\) regularly covers the subcomplex \(\overline{e'(\Pp)}\) of \(\Gamma \backslash \overline{X}\) with deck transformation group \(\Gamma_P\).  It thus is a finite free \(\Gamma_P\)-CW complex.  In fact \(\overline{e(\Pp)}\) is simply connected so that it can be identified with the universal covering of \(\overline{e'(\Pp)}\).  The nilpotent group \(\Gamma_{N_P}\) is elementary amenable and therefore of \(\det \geq 1\)-class \cite{Schick:L2-Determinant}.  It is moreover infinite because it acts cocompactly on the nilpotent Lie group \(N_P\).  This Lie group is diffeomorphic to a nonzero Euclidean space because \(\Pp \subset \Gg\) is proper.  By \fullref{thm:propofcellularl2}\,(\ref{item:propofcellularl2:aspherical}) the universal cover \(N_P\) of the finite CW-complex \(\Gamma_{N_P} \backslash N_P\) is \(L^2\)-acyclic and \(\rho^{(2)}(N_P; \mathcal{N}(\Gamma_{N_P})) = 0\).  The canonical base point \(K_P \in \overline{X}_\Pp\) and \fullref{prop:closureofep} define an inclusion \(N_P \subset \overline{e(\Pp)}\).  Applying \fullref{lemma:injectivefundamentalgroup} to \(N_P \subset \overline{e(\Pp)}\) and \(\Gamma_{N_P} \subset \Gamma_P\) shows that the fiber bundle \(\overline{e'(\Pp)}\) of \fullref{thm:gammaepfiberbundle} satisfies the conditions of \fullref{thm:propofcellularl2}\,(\ref{item:propofcellularl2:fiberbundle}).  We conclude that \(\overline{e(\Pp)}\) is \(\det\)-\(L^2\)-acyclic and
\[ \rho^{(2)}(\overline{e(\Pp)}, \mathcal{N}(\Gamma_P)) = \chi(\Gamma_{M_\Pp} \backslash \overline{X}_\Pp) \,\rho^{(2)}(N_P; \mathcal{N}(\Gamma_{N_P})) = 0. \proved \]
\end{proof}

\begin{proof}[Proof of \fullref{thm:l2torsionlattices}]
Fix a torsion-free arithmetic subgroup \(\Gamma \subset \Gg(\Q)\).  As remarked, the bordification \(\overline{X}\) is \(\det\)-\(L^2\)-acyclic by \fullref{thm:l2bettiofgenerallattices} because \(\delta(G) > 0\).  Since \(\overline{X}\) is even-dimensional, \fullref{lemma:l2torsionmanifoldwithboundary} says that the boundary \(\partial \overline{X}\) is \(\det\)-\(L^2\)-acyclic and we have proven the theorem when we can show \(\rho^{(2)}(\partial \overline{X}; \mathcal{N}(\Gamma)) = 0\).  To this end consider the space \(Y_k = \coprod_{\textup{s-rank}(\Pp) = k} e(\Pp)\) for \(k = 1, \ldots, \rank_\Q(\Gg)\), the coproduct of all boundary components \(e(\Pp)\) of rational parabolic subgroups \(\Pp \subset \Gg\) with split rank \(k\).  The usual action given in \eqref{eqn:actiononbordification} defines a free proper action of \(\Gamma\) on \(Y_k\) because the split rank of a rational parabolic subgroup is invariant under conjugation with elements in \(\Gg(\Q)\).  This action extends uniquely to a free proper action on the coproduct \(\overline{Y_k} = \coprod_{\textup{s-rank}(\Pp) = k} \overline{e(\Pp)}\) of closed boundary components because \(Y_k \subset \overline{Y_k}\) is dense.  The canonical \(\Gamma\)-equivariant map \(\overline{Y_k} \rightarrow \overline{X}\) lies in the pullback diagram
\[
\begin{xy}
\xymatrix{
\overline{Y_k} \ar[r] \ar[d] & \overline{X} \ar[d] \\
\Gamma \backslash \overline{Y_k} \ar[r] & \Gamma \backslash \overline{X}.
}
\end{xy}
\]
By \fullref{prop:decompofborelserre}, we have a finite system of representatives of \(\Gamma\)-conjugacy classes of rational parabolic subgroups of \(\Gg\).  Let \(\Pp_1^k, \ldots, \Pp_{r_k}^k\) be an ordering of the subsystem of rational parabolic subgroups with split rank \(k\).  Then \(\Gamma \backslash \overline{Y_k} = \coprod_{i=1}^{r_k} \overline{e'(\Pp_i^k)}\).  We apply \fullref{lemma:injectivefundamentalgroup} to each inclusion \(\overline{e(\Pp_i^k)} \subset \overline{X}\) and \(\Gamma_{P_i^k} \leq \Gamma\) to conclude that \(\overline{Y_k} = \coprod_{i=1}^{r_k} \overline{e(\Pp_i^k)} \times_{\Gamma_{\Pp_i^k}} \Gamma\).  Since every space \(\overline{e(\Pp_i^k)} \times_{\Gamma_{\Pp_i^k}} \Gamma\) is a \(\Gamma\)-invariant subcomplex of \(\partial \overline{X}\), this endows \(\overline{Y_k}\) with the structure of a finite free \(\Gamma\)-CW complex such that the equivariant map \(\overline{Y_k} \rightarrow \partial \overline{X}\) is cellular.  By the induction principle for \(L^2\)-torsion \cite{Lueck:L2-Invariants}*{Theorem~3.93(6) p.\,162} and \fullref{prop:l2torsionofboundarycomponent} \(\overline{Y_k}\) is \(\det\)-\(L^2\)-acyclic and
\[ \rho^{(2)}(\overline{Y_k}; \mathcal{N}(\Gamma)) = \sum\limits_{i=1}^{r_k} \rho^{(2)}(\overline{e(\Pp_i^k)} \times_{\Gamma_{\!P_i^k}} \!\Gamma; \mathcal{N}(\Gamma)) = \sum\limits_{i=1}^{r_k} \rho^{(2)} (\overline{e(\Pp_i^k)}; \mathcal{N}(\Gamma_{P_i^k})) = 0. \]
From \fullref{thm:propofcellularl2}\,(\ref{item:propofcellularl2:poincare}) we obtain that also \((\overline{Y}_k, \partial \overline{Y}_k)\) is \(\det\)-\(L^2\)-acyclic, so that the boundary \(\partial \overline{Y_k} = \overline{Y_k} \setminus Y_k\) is \(\det\)-\(L^2\)-acyclic by \cite{Lueck:L2-Invariants}*{Theorem 1.21, p.\,27}.  \fullref{lemma:l2torsionmanifoldwithboundary} says moreover that \(\rho^{(2)}(\partial \overline{Y_k}; \mathcal{N}(\Gamma)) = 0\) if \(\overline{Y_k}\) is even-dimensional.  But the same is true if \(\overline{Y_k}\) is odd-dimensional because of \fullref{thm:propofcellularl2}\,(\ref{item:propofcellularl2:poincare}).  Consider the \(\Gamma\)-CW subcomplexes \(\overline{X_k} = \bigcupdot_{\textup{s-rank}(\Pp) \geq k} \,e(\Pp)\) of \(\overline{X}\) where \(k = 1, \ldots, \rank_\Q(\Gg)\).  It follows from \eqref{eqn:intersectionofboundarycomponents} that they can be constructed inductively as pushouts of finite free \(\Gamma\)-CW complexes
\begin{equation}
\label{eq:pushout}
\begin{xy}
\xymatrix{
\partial \overline{Y_k} \ar[r] \ar[d] & \overline{X}_{k+1} \ar[d] \\
\overline{Y_k} \ar[r] & \overline{X}_k.
}
\end{xy}
\end{equation}
The beginning of the induction is the disjoint union \(\overline{X}_{\rank_\Q(\Gg)} = \bigcupdot_{\Pp \text{ min.}} \,e(\Pp)\) within \(\overline{X}\).  Since \(e(\Pp)\) is closed if \(\Pp\) is minimal, we observe as in the proof of \fullref{thm:novikovshubinqrankone} that in fact \(\overline{X}_{\rank_\Q(\Gg)} = \coprod_{\Pp \text{ min.}} \overline{e(\Pp)} = \overline{Y}_{\rank_\Q \Gg}\).  Therefore \fullref{lemma:l2torsionpushouts} verifies that each \(\overline{X_k}\) is \(\det\)-\(L^2\)-acyclic and \(\rho^{(2)}(\overline{X_k}; \mathcal{N}(\Gamma)) = 0\).  This proves the theorem because \(\overline{X_1} = \partial \overline{X}\).
\end{proof}

A group \(\Lambda\) has \emph{type \textit{F}}, if it possesses a finite CW model for \(B \Lambda\).  The Euler characteristic of a type \textit{F} group is defined by \(\chi(\Lambda) = \chi(B \Lambda)\).  A slight generalization of this is due to C.\,T.\,C.\,Wall \cite{Wall:RationalEuler}.  If \(\Lambda\) virtually has type \textit{F}, its \emph{virtual Euler characteristic} is given by \(\chi_{\textup{virt}}(\Lambda) = \frac{\chi(\Lambda')}{[\Lambda : \Lambda']}\) for a finite index subgroup \(\Lambda'\) with finite CW model for \(B\Lambda'\).  This is well-defined because the Euler characteristic is multiplicative under finite coverings.  Since the \(L^2\)-torsion in many respects behaves like an odd-dimensional Euler characteristic, we want to define its virtual version as well.  If a group \(\Gamma\) is virtually \(\det\)-\(L^2\)-acyclic, we define \(\rho^{(2)}_{\textup{virt}}(\Gamma) = \frac{\rho^{(2)}(\Gamma')}{[\Gamma : \Gamma']}\) for a finite index subgroup \(\Gamma'\) with finite \(\det\)-\(L^2\)-acyclic \(\Gamma'\)-CW model for \(E\Gamma'\).  Again this is well-defined because \(\rho^{(2)}\) is multiplicative under finite coverings.

\begin{lemma} \label{lemma:l2torsionproductformula}
Let \(\Lambda\) be virtually of type \textit{F} and let \(\Gamma\) be virtually \(\det\)-\(L^2\)-acyclic.  Then \(\Lambda \times \Gamma\) is virtually \(\det\)-\(L^2\)-acyclic and
\[ \rho^{(2)}_{\textup{virt}}(\Lambda \times \Gamma) = \chi_{\textup{virt}}(\Lambda) \cdot \rho^{(2)}_{\textup{virt}}(\Gamma). \]
\end{lemma}

\begin{proof}
Let \(\Lambda' \leq \Lambda\) and \(\Gamma' \leq \Gamma\) be finite index subgroups with finite classifying spaces such that \(E\Gamma'\) is \(\det\)-\(L^2\)-acyclic and apply \fullref{thm:propofcellularl2}\,(\ref{item:propofcellularl2:fiberbundle}) to the trivial fiber bundle \(B\Gamma' \rightarrow B(\Lambda' \times \Gamma') = B\Lambda' \times B\Gamma' \rightarrow B\Lambda'\).
\end{proof}

\medskip
{\bf \fullref{thm:virtuall2torsion}}\qua
{\sl Let \(G\) be a connected semisimple linear Lie group with positive, even deficiency.  Then every lattice \(\Gamma \subset G\) is virtually \(\det\)-\(L^2\)-acyclic and}
\[ \rho^{(2)}_{\textup{virt}}(\Gamma) = 0. \]

\begin{proof}
By Selberg's Lemma there exists a finite index subgroup \(\Gamma' \subset \Gamma\) which is torsion-free.  Thus \(\Gamma'\) can neither meet any compact factor nor the center of \(G\) which is finite because \(G\) is linear.  Therefore we may assume that \(G\) has trivial center and no compact factors.  Suppose \(\Gamma'\) was reducible.  By \cite{WitteMorris:ArithGroups}*{Proposition~4.24, p.\,48} we have a direct product decomposition \(G = G_1 \times \cdots \times G_r\) with \(r \geq 2\) such that \(\Gamma'\) is commensurable with \(\Gamma'_1 \times \cdots \times \Gamma'_r\) where \(\Gamma'_i = G_i \cap \Gamma'\) is irreducible in \(G_i\) for each \(i\).  Again by Selberg's Lemma we may assume that \(\Gamma'_1 \times \cdots \times \Gamma'_r\) is torsion-free.  If \(\rank_\R(G_i) = 1\), then \(\Gamma_i\) is type \textit{F}, for example by a compactification of H.\,Kang \cite{Kang:Cofinite}.  If \(\rank_\R(G_i) > 1\), then \(\Gamma_i\) is virtually type \textit{F} by Margulis arithmeticity, \fullref{thm:margulisarithmeticity}, and the Borel-Serre compactification.  Therefore, and by \fullref{thm:l2bettiofgenerallattices}, \(\Gamma'_1 \times \cdots \times \Gamma'_r\) and thus \(\Gamma\) is virtually \(\det\)-\(L^2\)-acyclic.  Thus we may assume that \(\Gamma'_1 \times \cdots \times \Gamma'_r\) is honestly \(\det\)-\(L^2\)-acyclic and we have to show that \(\rho^{(2)}(\Gamma'_1 \times \cdots \times \Gamma'_r) = 0\).

Since \(\delta(G) > 0\), there must be a factor \(G_{i_0}\) with \(\delta(G_{i_0}) > 0\).  Let \(H\) be the product of the remaining factors \(G_i\) and let \(\Gamma_H\) be the product of the corresponding irreducible lattices \(\Gamma_i\).  If \(\delta(H) > 0\), then \(\Gamma_H\) is \(\det\)-\(L^2\)-acyclic by \fullref{thm:l2bettiofgenerallattices} and \(\rho^{(2)}(\Gamma'_1 \times \cdots \times \Gamma'_r) = \rho^{(2)}(\Gamma'_{i_0} \times \Gamma_H) = 0\) by \fullref{lemma:l2torsionproductformula} because \(\chi(\Gamma'_{i_0}) = 0\) by \fullref{thm:propofcellularl2}\,(\ref{item:propofcellularl2:fiberbundle}).  If \(\delta(H) = 0\), then \(\delta(G_{i_0})\) is even, and \fullref{lemma:l2torsionproductformula} says that \(\rho^{(2)}(\Gamma_H \times \Gamma'_{i_0} ) = \chi(\Gamma_H) \rho^{(2)}(\Gamma'_{i_0})\).  So we may assume that the original \(\Gamma'\) was irreducible.  We have \(\rank_\R(G) \geq \delta(G) \geq 2\) as follows from \cite{Borel-Wallach:ContinuousCohomology}*{Section III.4, Formula (3), p.\,99}.  By Margulis arithmeticity, \fullref{thm:margulisarithmeticity}, \(\Gamma'\) is abstractly commensurable to \(\Hh(\Z)\) for a connected semisimple linear algebraic \(\Q\)-group \(\Hh\).  Moreover \(\delta(\Hh(\R)) = \delta(G)\) because \(\Hh(\R)\) and \(G\) define isometric symmetric spaces.  \fullref{thm:l2torsionlattices} completes the proof.
\end{proof}

It remains to give some details for our application to the L\"uck--Sauer--Wegner conjecture.

\medskip
{\bf \fullref{thm:l2torsionmeasureequivalence}}\qua
{\sl Let \(\mathcal{L}^{\textup{even}}\) be the class of \(\det\)-\(L^2\)-acyclic groups that are measure equivalent to a lattice in a connected simple linear Lie group with even deficiency.  Then \textup{\fullref{conj:luecksauerwegner}} holds true for \(\mathcal{L}^{\textup{even}}\).}

\begin{proof}
Let \(\Gamma \in \mathcal{L}^{\textup{even}}\) be measure equivalent to \(\Lambda \subset G\) with G as stated.  Then \(\delta(G) > 0\) by \cite{Gaboriau:Invariantsl2}*{Th\'eor\`eme 6.3, p.\,95} because \(\Gamma\) is \(L^2\)-acyclic by assumption.  Since \(\Gamma\) has a finite \(B\Gamma\), it is of necessity torsion-free so that \(\Gamma\) is a lattice in \(\operatorname{Ad} G\) by \cite{Furman:MERigidity}*{Theorem~3.1, p.\,1062}.  \fullref{thm:virtuall2torsion} applied to \(\Gamma \subset \operatorname{Ad} G\) completes the proof.
\end{proof}

%%%%%%%%%%%%%%%%%%%%   End of main body of article
%
%                             References
%
%   BiBTeX users uncomment the following line:
%
%\bibliographystyle{gtart}
%

\end{document}